%% file: IFE_moving.tex
\documentclass[onefignum,onetabnum]{siamart190516}

\usepackage{geometry}
\allowdisplaybreaks[4]
\input{ex_shared}
\usepackage{subcaption}
\newcommand{\bft}{{\bf t}}
\newcommand{\bfn}{{\bf n}}

\newcommand{\vertiii}[1]{{\left\vert\kern-0.25ex\left\vert\kern-0.25ex\left\vert #1
    \right\vert\kern-0.25ex\right\vert\kern-0.25ex\right\vert}}
    \newcommand{\vertii}[1]{{\left\vert\kern-0.25ex\left\vert #1
    \right\vert\kern-0.25ex\right\vert}}
\newcommand{\verti}[1]{{\left\vert #1
    \right\vert}}
\ifpdf
\hypersetup{
  pdftitle={Solving Parabolic Moving Interface Problems with Dynamical Immersed Spaces on Unfitted Meshes: Fully discrete Analysis},
  pdfauthor={Ruchi Guo}
}
\headers{IFEM for Moving Interface Problems}{R. Guo}

\title{Solving Parabolic Moving Interface Problems with Dynamical Immersed Spaces on Unfitted Meshes: Fully discrete Analysis}

\author{Ruchi Guo\thanks{Department of Mathematics, The Ohio State University, Columbus, OH 43210
  \email{guo.1778@osu.edu}.}
}

\DeclareFontEncoding{FMS}{}{}
\DeclareFontSubstitution{FMS}{futm}{m}{n}
\DeclareFontEncoding{FMX}{}{}
\DeclareFontSubstitution{FMX}{futm}{m}{n}
\DeclareSymbolFont{fouriersymbols}{FMS}{futm}{m}{n}
\DeclareSymbolFont{fourierlargesymbols}{FMX}{futm}{m}{n}
\DeclareMathDelimiter{\tbar}{\mathord}{fouriersymbols}{152}{fourierlargesymbols}{147}

\begin{document}

\maketitle

\begin{abstract}
Immersed finite element (IFE) methods are a group of long-existing numerical methods for solving interface problems on unfitted meshes. A core argument of the methods is to avoid mesh regeneration procedure when solving moving interface problems. Despite the various applications in moving interface problems, a complete theoretical study on the convergence behavior is still missing. This research is devoted to close the gap between numerical experiments and theory. We present the first fully discrete analysis including the stability and optimal error estimates for a backward Euler IFE method for solving parabolic moving interface problems. Numerical results are also presented to validate the analysis.
\end{abstract}



\section{Introduction}
\label{sec:introduction}
Interface problems raise from various models that involve multiple materials with different chemical or physical properties. In these models, the interface geometry itself may involve certain dynamics, i.e., the whole or portion of the interface evolve. Let $\Omega\subseteq\mathbb{R}^2$ be a fixed domain and let $\Gamma(t)$ be an evolving interface curve partitioning $\Omega$ into two subdomains $\Omega^-(t)$ and $\Omega^+(t)$ on a time interval $[0,T]$. Suppose there is certain velocity field $\mathcal{V}(X,t)$ guiding the movement of the interface curve, i.e.,
\begin{equation}
\label{velocity}
\frac{dX}{dt} = \mathcal{V}(X,t) ~~~~ X\in \Gamma(t).
\end{equation}
We further let $\beta$ be a piecewise constant function such that
\begin{equation*}
\beta(X,t)=
\left\{\begin{array}{cc}
\beta^- & \text{for} \; X\in \Omega^-(t) ,\\
\beta^+ & \text{for} \; X\in \Omega^+(t),
\end{array}\right.
\end{equation*}
which is associated with some physical or chemical properties of the materials occupying each subdomain. In this article, we consider the following parabolic interface model
\begin{subequations}\label{model}
\begin{align}
\label{inter_PDE}
\partial_t u -\nabla\cdot(\beta\nabla u)=f, \;\;\;\; & \text{in} \; \Omega = \Omega^-  \cup \Omega^+, ~~ t\in [0,T], \\
 u(\cdot,t)=0, \;\;\;\; &\text{on} \; \partial\Omega, ~~ t\in [0,T], \\
 u(\cdot,0) = u_0 ,\;\;\;\;  &\text{in} \; \Omega = \Omega^-  \cup \Omega^+.
\end{align}
The following jump conditions are imposed on the interface $\Gamma(t)$:
\begin{align}
[u]_{\Gamma(t)} &:=u^-|_{\Gamma(t)} - u^+|_{\Gamma(t)}= 0,  ~~ t\in [0,T], \label{jump_cond_1} \\
\big[\beta \nabla u\cdot \mathbf{n}\big]_{\Gamma(t)} &:=\beta^- \nabla u^-\cdot \mathbf{n}|_{\Gamma(t)} - \beta^+ \nabla u^+\cdot \mathbf{n}|_{\Gamma(t)} = 0, ~~ t\in [0,T], \label{jump_cond_2}
\end{align}
\end{subequations}
in which $\mathbf{ n}$ is the unit normal vector to $\Gamma(t)$. Here we only discuss the homogeneous jump conditions in the analysis, and the non-homogeneous jumps can be simply handled by the enriched functions as discussed by Babu\v{s}ka et al in \cite{2020AdjeridBabukaGuoLin}. The parabolic interface model in \eqref{velocity} and \eqref{model} widely appear in many applications. A well-known example is the Stefan problem \cite{1993Almgren,1997ChenMerrimanOsherSmereka} to model solidification process where $u$ represents the temperature and the velocity $\mathcal{V}$ is computed by the flux of temperature across the interface. It also appears in the Burton-Cabrera-Frank-type model for epitaxial growth of thin films \cite{2003CaflischLi} where $u$ denotes the adatom density and the velocity $\mathcal{V}$ depends on the flux of adatom density across the interface. Another example can be found in using shape optimization methodology to reconstruct inclusions governed by heat equations \cite{2013HarbrechtTausch}. In this case, the velocity is associated with the direction that shape functionals have the greatest descent rate, and computed though adjoint equations. 

It is well-known that moving interface problems may cause challenges to simulation since the modeling domain itself is evolving. If traditional finite element methods are applied, meshes have to be generated to fit the interface, and thus have to be moving or regenerated according to interface movement; otherwise the accuracy of the numerical solutions can be destroyed \cite{2000BabuskaOsborn}. The general principal is to reduce the frequency of completely remeshing procedure as much as possible, since remeshing could be troublesome, time-consuming and introduce projection or interpolation errors. There have been many moving mesh methods proposed in the literature such as the early researches \cite{1967Winslow} by Winslow based on solving elliptic-type PDEs to generate mapping for mesh generation and \cite{1992TezduyarBehrLiou} based on time-space formulation. Another typical example is the so called arbitrary Lagrangian-Eulerian (ALE) method \cite{1981HughesLiuZimmermann,2020LanRamirezSun} to solve fluid-structure-interaction (FSI) problems. In addition, we also refer readers to moving mesh methods based on Harmonic mappings \cite{2007DiLiTangZhang,2009HuLiTang} applied to diffusive interface models.


Alternatively, in order to completely remove the burden of mesh moving or remeshing procedure in the computation, numerical methods based on interface-independent unfitted meshes have evoked a lot of interests among many researchers in the past decades. To handle interface-cutting elements, a group of methods enforce the jump conditions in the computation scheme such as the immersed interface methods (IIM) \cite{1994LevequeLi,1997Li} based on the finite difference framework, and CutFEM \cite{2015BurmanClaus,2017HuangWuXiao} and fictitious domain methods \cite{2017WangSun} based on the finite element (FE) framework. In the context of FE methods (FEM), another group of methods attempt to use some specially designed shape functions to incorporate the jump information such as generalized FEM \cite{1983BabuskaOsborn}, multiscale FEM (MsFEM) \cite{2010ChuGrahamHou}, extended FEM (XFEM) \cite{2001DolbowMoesBelytschko} and immersed finite element (IFE) method to be discussed in this article. 



It is important to note that the theoretical analysis has been extensively studied for all these unfitted mesh methods on stationary interface problems, but the theoretical work on moving interface problems is rather limited in the literature. When interface evolves, an extra obstacle stems from the variation of approximation spaces and computation schemes in dynamics. For moving mesh methods, the analysis is based on the mesh-generation mapping between the fixed reference domain and the evolving physical domain, see \cite{1997HuangRussel} and particularly \cite{2017StefanThomas,2020LanRamirezSun} for interface problems. But this strategy is not suitable for unfitted mesh methods since the dynamics of approximation spaces is independent of the mesh. To address this issue, in \cite{2013LehrenfeldReusken} the authors considered a space-time discontinuous Galerkin method based on XFEM but only suboptimal convergence with respect to time can be obtained. The author in \cite{2013Zunino} studied a backward Euler XFEM but the analysis approach depends on certain strong assumptions on the interpolation operators, see (12)-(15) in that article. 

The core idea of IFE methods is to use piecewise polynomials on interface elements to capture the jump behavior of the exact solutions. IFE methods are especially attractive for moving interface problems not only because they can be used on unfitted meshes but also the IFE spaces are isomorphic to the standard FE spaces defined on the same mesh, namely the degrees of freedom also keep unchanged in dynamics. Since the IFE method was first introduced in \cite{1998Li}, it has been applied to solve various moving interface problems. For instance, the authors in \cite{2015AdjeridChaabaneLin,2018AdjeridChaabaneLinYue} developed the IFE method for incompressible interfacial flows governed by Stokes equations and applied it to simulate drop behavior in shear and extensional flow. The authors in \cite{2019AdjeridMoon} investigated the IFE method for acoustic wave propagation problems where the simulation is conducted for an air bubble moving in water. A simulation for a moving object by IFE methods in electromagnetic field was conducted in \cite{2018BaiCaoHeLiuYang}. An IFE-based shape optimization method for geometric inverse problems was proposed in \cite{2018GuoLinLinElasto}. As far as we know, the numerical exploration for convergence behavior, without an error analysis, can be only found in \cite{2013HeLinLinZhang,2013LinLinZhang1} for parabolic interface problems. Despite these applications and numerical exploration, the theoretical analysis still remains open. 

Roughly speaking, the key difficulty in the analysis of IFE methods comes from the insufficient regularity of IFE functions including the kink across interface and discontinuities across interface edges. Namely the local IFE spaces on interface elements are only in $H^1$, and the global IFE space is not even $H^1$-conforming which all are weaker than the standard FE spaces. Thus many critical results such as trace/inverse inequalities and interpolation/projection errors can not be proved by standard techniques. For static interface problems (no time), a series of articles have built a systematic analysis framework \cite{2020AdjeridBabukaGuoLin,2019GuoLin,2003LiLinWu,2015LinLinZhang} which can establish those inequalities and estimates for IFE functions. These results are then employed in \cite{2015LinYangZhang,2020LinZhuang} to analyze the IFE methods for time-dependent problems but with the stationary interface. However there is still a gap between the analysis for stationary and unstationary interface problems due to the discontinuities of IFE spaces not only in spatial direction but also the temporal direction. 

Thanks to the isomorphism between the IFE and FE spaces, we are able to construct a uniform weak form throughout the dynamics and restrict all the variations only to the IFE spaces. This idea motivates us to reconsider the discontinuities of IFE spaces along the temporal direction from the perspective of time stepping discontinuous Galerkin method \cite{1985ErikssonJohnsonThomee}, and thus recast the time stepping IFE scheme into the framework of time-dependent adaptive methods \cite{1991ErikssonJohnsonI,1995ErikssonClaes}. The isomorphism also enables us to show that the IFE spaces share some nice properties of their FE images such as the trace inequality which is non-trivial since the IFE spaces are not $H^1$-conforming. By these preparation we present the first fully discrete optimal error estimates for a backward Euler IFE method solving the parabolic interface model \eqref{model}. 

This article consists of five additional sections. In the next section, we set up some basic notations and assumptions. In Section \ref{sec:ife_discret}, we recall the IFE spatial discretization and develop the backward Euler method. In Section \ref{sec:pre_est} we prepare some fundamental estimates. The fully discrete error estimates are presented in Section \ref{sec:error_est}. Some numerical experiments are shown in the last section to validate the analysis.




\section{Notations and Assumptions}

Throughout this article, we let $\mathcal{T}_h$ be a family of shape regular and quasi-uniform triangular partition of $\Omega$ which is independent of the evolving interface $\Gamma(t)$. For each $T\in \mathcal{T}_h$, we let $h_T$ be its diameter and define $h=\max_{T\in\mathcal{T}_h}h_T$ as the mesh size. Also we let $\mathcal{E}_h$, $\mathring{\mathcal{E}}_h$ and $\mathcal{N}_h$ be the collection of edges, interior edges and mesh nodes, respectively. We denote all the elements intersecting with $\Gamma(t)$ by $\mathcal{T}^i_h(t)$, i.e., the collection of interface elements. Similarly, we define the collection of interface edges as $\mathcal{E}^i_h(t)$. We emphasize that these two collections are all time-dependent, i.e., they depend on the interface location at $t$. In the analysis we employ a generic constant $C$ which is independent of mesh size and the interface location relative to the mesh. 


For each manifold $\omega\subseteq\Omega$, we define $H^k(\omega)$ as the standard Hilbert space with the norm $\|\cdot\|_{H^k(\omega)}$, and define the time-dependent Bochner space $H^l(0,T;H^k(\omega))$ with the norm $\|\cdot\|_{H^l(0,T;H^k(\omega))}$. If $|\omega\cap\Gamma|\neq0$, we let $\omega^{\pm}=\Omega^{\pm}\cap\omega$, define the split Hilbert space $H^k(\omega^-\cup\omega^+)=H^k(\omega^-\cup\omega^+,t)=\{v:v\in H^k(\omega^{\pm}(t))\}$ and further define the space involving the jump conditions:
\begin{equation}
\label{split_space}
\widetilde{H}^k(\omega,t) = \{ v\in H^k(\omega^{\pm}(t))~:~ [v]_{\Gamma(t)}=0,~ [\beta\nabla v\cdot\bfn]_{\Gamma(t)} = 0 \}
\end{equation}
where we assume $k>3/2$ such that the traces are well-defined, and there clearly holds $\widetilde{H}^k(\omega)\subseteq H^1(\omega)\cap H^k(\omega^-\cup\omega^+)$. Note that the two spaces above are all time-dependent due to $\Gamma(t)$, but we shall drop $t$ if there is no cause of confusion. Then the norms associated with $\widetilde{H}^k(\omega)$ and $H(\omega^-\cup\omega^+)$ are understood as $\|\cdot\|^2_{H^k(\omega)}=\|\cdot\|^2_{H^k(\omega^+)}+\|\cdot\|^2_{H^k(\omega^-)}$. We also denote $H^k_0(\omega)$, $H^k_0(\omega^-\cup\omega^+)$ and $\widetilde{H}^k_0(\omega)$ as the subspaces with zero trace on $\partial \omega$. Furthermore, on the mesh $\mathcal{T}_h$, we define a underling space containing all the approximation spaces considered in this article
\begin{equation}
\begin{split}
\label{split_space_2}
W_h = \{  v_h \in L^2(\Omega)~:~ &v_h|_T\in H^1(T) ~ \forall T\in\mathcal{T}_h ~ \text{and} ~ v_h ~ \text{is continuous at each} ~ X\in\mathcal{N}_h, ~ v_h|_{\partial\Omega}=0, \\
 &  \nabla v_h\cdot\bfn ~ \text{is well-defined on each} ~ e\in \mathcal{E}_h ~\text{and belong to} ~L^2(e) \}.
\end{split}
\end{equation} 
Furthermore we define $\mathbb{P}_k(\omega)$ as the polynomial space with the degree not greater than $k$ where $k$ is any non-negative integer. We also define $(\cdot,\cdot)_{L^2(\omega)}$ as the standard $L^2$ inner product on ${\omega}$.

At each $t\in [0,T]$, we assume $\Gamma(t)$ is a sufficiently smooth simple Jordan curve, namely it does not intersect itself. For simplicity, we also assume $\Gamma(t)$ does not touch the boundary. Furthermore we assume the interface only intersects an element $T$ with exactly two points locating on different edges as shown in Figure \ref{fig:interf_elem}. This assumption is widely used for many unfitted mesh methods on stationary interface problems, see \cite{2015BurmanClaus,2016GuoLin,1994LevequeLi} and the reference therein. We then connect all these intersection points to form a polyline $\Gamma_h(t)$ as the linear approximation to $\Gamma(t)$ shown in Figure \ref{fig:mesh}. An alternative way to construct $\Gamma_h(t)$ employs the level-set method \cite{2001OsherFedkiw} with piecewise linear elements. Namely, for a level-set representation $\varphi(t)$ of $\Gamma(t)$, we let $\varphi_h(t)$ be its continuous piecewise linear approximation computed by some algorithm, and then define $\Gamma_h(t)$ as the zero level-set of $\varphi_h(t)$. Here $\Gamma_h(t)$ exactly satisfies the assumption above since the $\varphi_h(t)$ is piecewise linear; but the intersection points of $\Gamma_h(t)$ are in general different from  those of $\Gamma(t)$. We emphasize that these linear approximation $\Gamma_h(t)$ have $\mathcal{O}(h^2)$ geometric accuracy to the original interface $\Gamma(t)$ which is sufficient for the linear finite element method considered in this article. As for higher order methods, a higher order geometric approximation is needed and we refer readers to \cite{2010LiMelenkWohlmuthZou} for more details. IFE methods can be also applied to solve stationary interface problems with arbitrary high order accuracy \cite{2019GuoLin}.

\begin{figure}[h]
\centering
\begin{minipage}{.42\textwidth}
  \centering
  \includegraphics[width=2.2in]{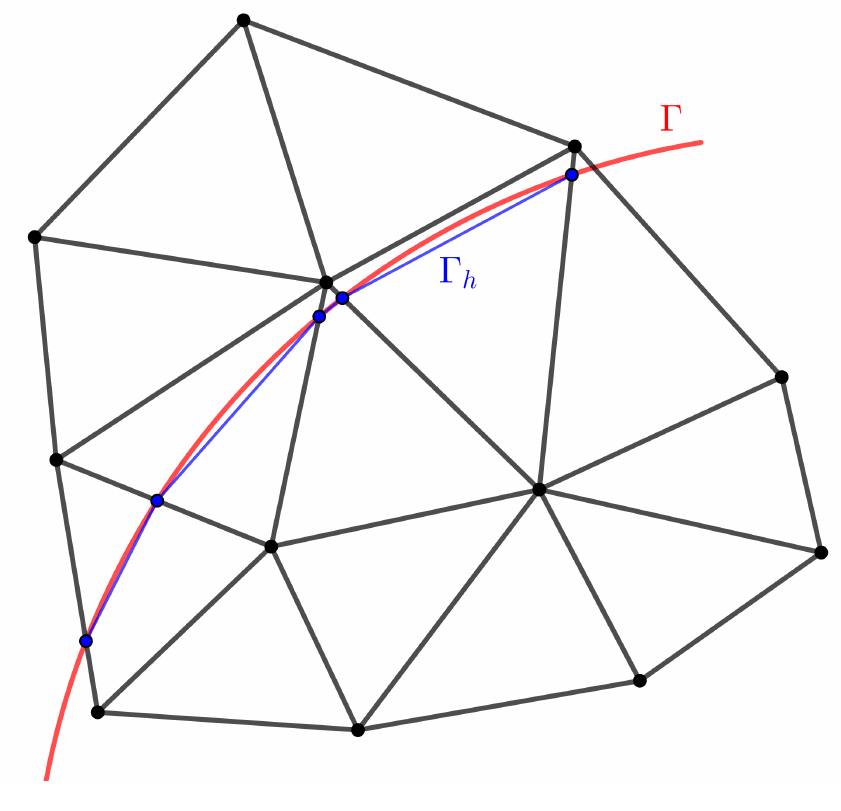}
  \caption{A unfitted mesh}
  \label{fig:mesh}
\end{minipage}
\hspace{2cm}
\begin{minipage}{.4\textwidth}
  \centering
  \ \includegraphics[width=2in]{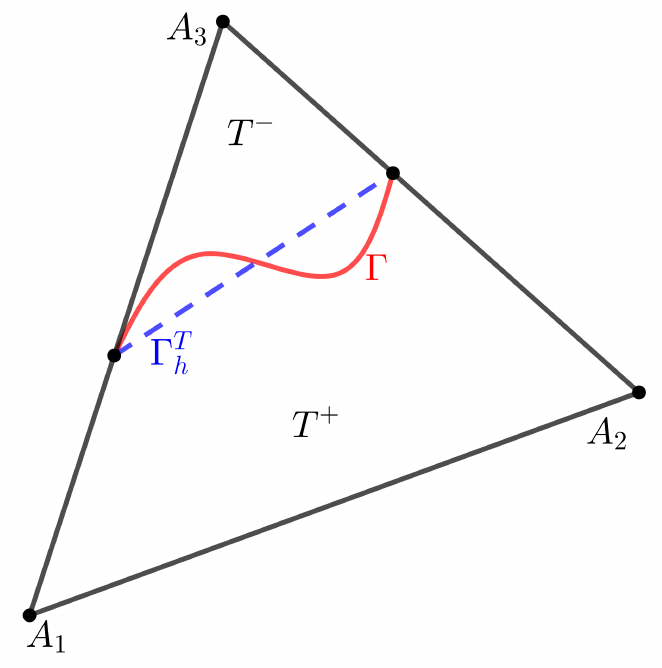}
  \caption{An interface element}
  \label{fig:interf_elem}
  \end{minipage}
\end{figure}


To end this section, we recall the Reynolds Transport Theorem \cite{1903Reynolds} in the context of fluid dynamics (a similar one referred as shape derivative formula can be found in the context of shape calculus, see (2.168) in \cite{J.Sokolowski_J.-P.Zolesio_1992}). Suppose the velocity $\mathcal{V}(t)$ is sufficiently smooth on $\Gamma$, then given any differentiable functional defined in terms of integral on $\Omega^{\pm}$
\begin{equation}
\label{functional_1}
\mathcal{J}^{\pm}(t) = \int_{\Omega^{\pm}(t)} j(t,X) dX
\end{equation}
its temporal derivative with respect to the $\mathcal{V}$ direction can be calculated by
\begin{equation}
\label{functional_2}
\frac{d}{dt} \mathcal{J}^{\pm}(t) = \int_{\Omega^{\pm}(t)} \partial_t j(t,X) dX + \int_{\Gamma(t)} j \mathcal{V}\cdot\bfn ds,
\end{equation}
where $\bfn$ is the norm vector to $\Gamma$ and outward to to $\Omega^{\pm}$.


\section{IFE Discretization}
\label{sec:ife_discret}

In this section, we first describe a linear IFE method for the spatial approximation, and then present a backward Euler method for the temporal approximation.

\subsection{Spatial Discretization}
The core of IFE methods is the so called IFE functions to approximate the jump conditions. At each $t$, let's define $\Gamma^T_{h}(t)=\Gamma_h(t)\cap T$ for every interface element $T\in\mathcal{T}^i_h(t)$ which is simply the segment connecting the intersection points shown in Figure \ref{fig:interf_elem}, and without causing any confusion we let $\Gamma^T_{h}(t)$ divide $T$ into $T^{\pm}(t)$. Then on each interface element $T\in\mathcal{T}^i_h(t)$ with the vertices $A_j$, $j=1,2,3$, the linear IFE space consists of piecewise linear polynomials such that they satisfy the jump conditions on $\Gamma^T_{h}(t)$, namely 
\begin{equation}
\begin{split}
\label{loc_IFE_spa}
S_{h,T}(t) = & \{ v_h~:~ v^{\pm}_h=v_h|_{T^{\pm}}\in \mathbb{P}_1(T^{\pm}), ~ [v_h]_{\Gamma_{h}^T(t)}=0, ~ [\beta \nabla v_h\cdot\bar{\bfn}]_{\Gamma_{h}^T(t)}=0 \} \\
= & \text{Span}\{ \psi_{1,T}, \psi_{2,T}, \psi_{3,T} \}
\end{split}
\end{equation}
where $\bar{\bfn}$ is the normal vector to $\Gamma^T_h(t)$, and $ \psi_{i,T}$, $i=1,2,3$ are the Lagrange-type shape functions satisfying
\begin{equation}
\label{unisolv}
\psi_{i,T}(A_j) = \delta_{ij} ~~~~ i,j=1,2,3.
\end{equation}
The unisolvence of these shape functions is guaranteed regardless of interface location and $\beta^{\pm}$, and we refer interested readers to Theorem 5.3 of \cite{2016GuoLin} for more details. On all the non-interface elements, the local IFE spaces are simply linear polynomial spaces, i.e., $S_{h,T}(t)=\mathbb{P}_1(T)$. We note that these local IFE spaces vary in dynamics since the interface is evolving. We define the global IFE space
\begin{equation}
\label{glob_IFE_spa}
S_h(t) = \{ v_h\in L^2(\Omega)~:~ v_h|_T\in S_{h,T}(t) ~ \forall T\in \mathcal{T}_h ~ v_h ~ \text{is continuous at } ~ X\in \mathcal{N}_h ~ \text{and} ~ v_h|_{\partial\Omega}=0 \}.
\end{equation}
Clearly we have $S_{h,T}(t)\subseteq H^1(T)$ but $S_h(t)$ is not $H^1$ conforming i.e., $S_h(t) \not\subset H^1(\Omega)$ because IFE functions may not be continuous across interface edges. We note that the global IFE space in \eqref{glob_IFE_spa} is isomorphic to the standard continuous piecewise linear FE space denoted by $\widetilde{S}_h$. To see this, let's define the standard nodal interpolation operator
\begin{equation}
\label{iso_map}
\mathcal{I}_h(t)~:~ W_h \longrightarrow S_h(t) ~~~ \text{such that} ~~ \mathcal{I}_h(t) v_h(X) = v_h(X) ~~ \forall X\in \mathcal{N}_h.
\end{equation}
Here we note that $\mathcal{I}_h(t)$ is time-dependent purely because its range $S_h(t)$ depends on time, but the manner of the definition itself keeps unchanged. We also define the local interpolation $\mathcal{I}_{h,T}=\mathcal{I}_h|_T$. Since $\widetilde{S}_h\subseteq W_h$, we have $\mathcal{I}_h(t)$ restricted on $\widetilde{S}_h$ exactly gives the isomorphism between these two spaces. For example, we plot an IFE function in $S_h$ in Figure \ref{fig:ife_fun}\subref{ife_fun_1} and its isomorphic image in $\widetilde{S}_h$ in Figure \ref{fig:ife_fun}\subref{ife_fun_2}. Besides, comparing these two functions, we can clearly see that the IFE function can capture details much better across the interface while the FE function just losses interface information, but away from the interface they are exactly the same. Zooming in the function in Figure \ref{fig:ife_fun}\subref{ife_fun_1} we can see the slight discontinuities on some interface edges in Figure \ref{fig:ife_fun}\subref{ife_fun_1_zoom}. Moreover it has been proved that the IFE functions/spaces share many nice properties similar to the standard FE functions such as optimal approximation capabilities, trace/inverse inequalities and uniform bounds. These properties are presented in a series of articles \cite{2016GuoLin,2004LiLinLinRogers,2015LinLinZhang}. For readers' sake, we shall recall these results here since they will be used for the analysis later. 

\begin{theorem}
There exists a constant $C$ such that for every interface $\Gamma(t)$ and interface element $T$
\begin{subequations}
\label{recall_theorem}
\begin{align}
   (\textbf{approximation capability}) ~~ &\| u - \mathcal{I}_h u \|_{L^2(\Omega)} + h |u - \mathcal{I}_h u|_{H^1(\Omega)} \le Ch^2 \| u \|_{H^2(\Omega)} ~~ \forall u\in \widetilde{H}^2(\Omega), \label{thm_appro_eq0} \\
   (\textbf{inverse inequality}) ~~ &\| \nabla v_h \|_{L^2(T)} \le Ch^{-1}_T \| v_h \|_{L^2(T)} ~~ \forall v_h\in S_{h,T}(t), \label{inver_inequa} \\
    (\textbf{trace inequality}) ~~& \| \nabla v_h \|_{L^2(e)} \le Ch^{-1/2}_T \| \nabla v_h \|_{L^2(T)} ~~ \forall v_h\in S_{h,T}(t) , \label{trace_inequa} \\
    (\textbf{boundedness}) ~~  &|  \psi_{i,T} |_{W^{j,\infty}(T)} \le Ch^{-j}_T,~~~j=0,1, ~ i=1,2,3.  \label{boundedness} 
\end{align}
\end{subequations}
\end{theorem}

\begin{figure}[h]
\centering
\begin{subfigure}{.3\textwidth}
     \includegraphics[width=2in]{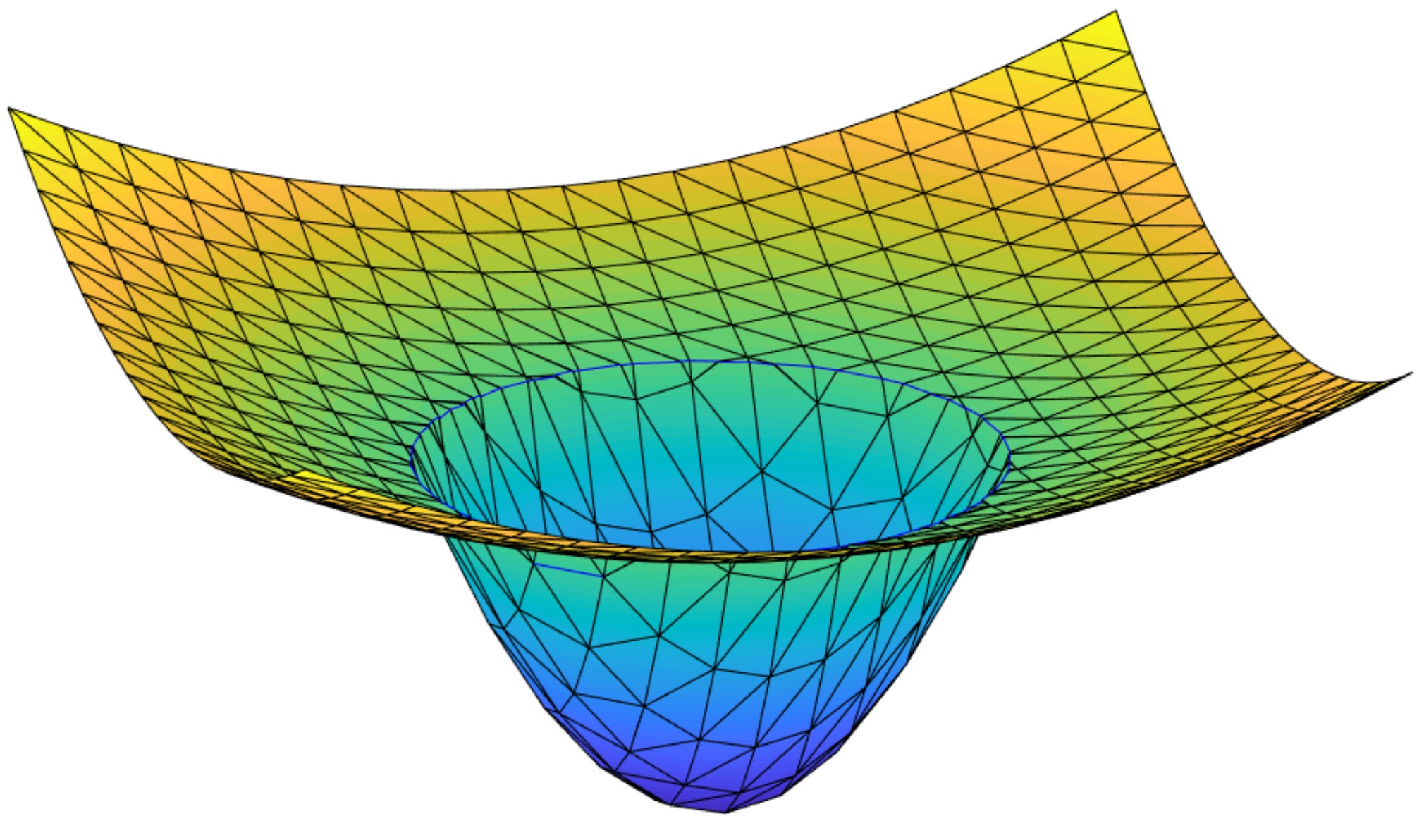}
     \caption{A global function in $S_h(t)$}
     \label{ife_fun_1} 
\end{subfigure}
~
\begin{subfigure}{.3\textwidth}
     \includegraphics[width=2in]{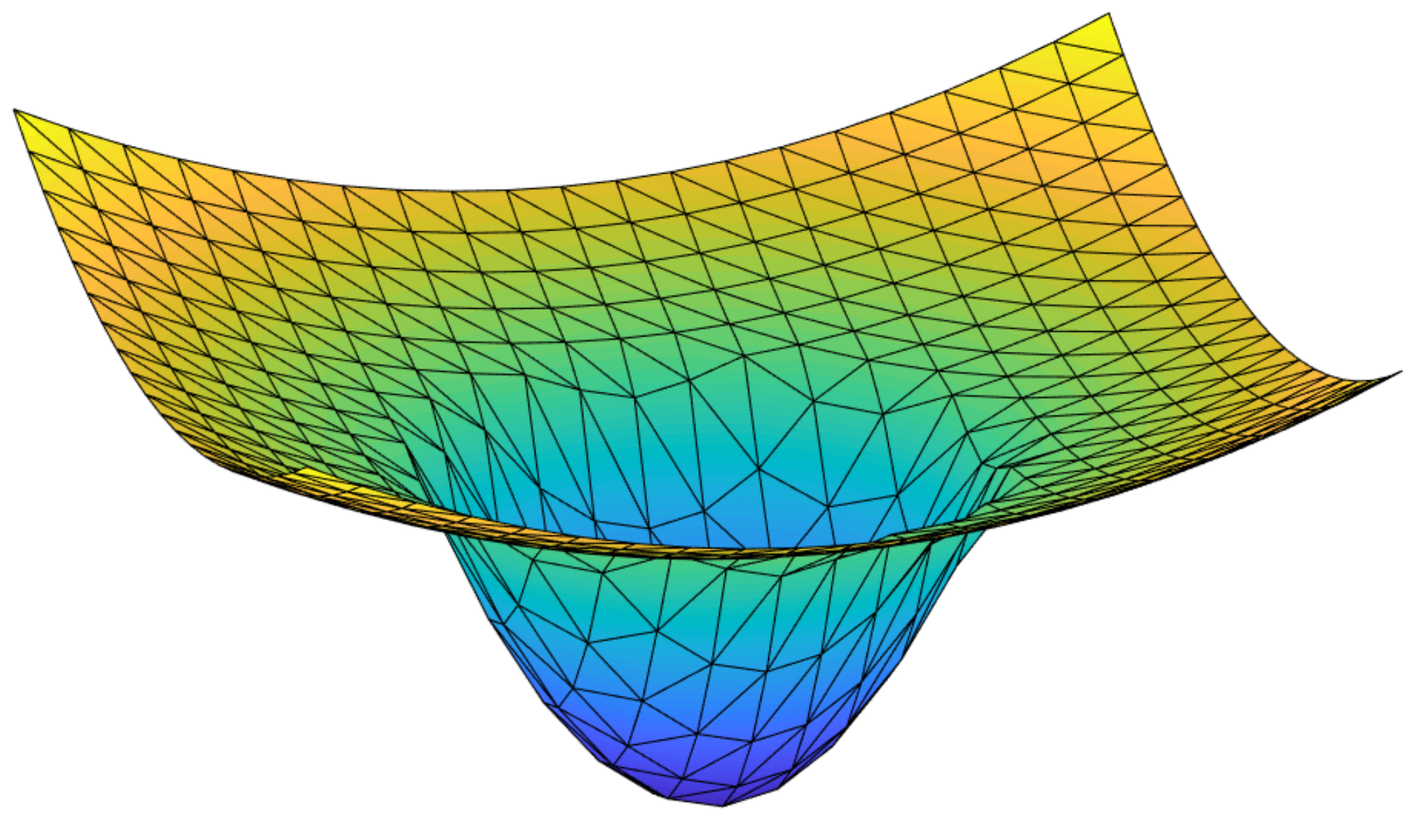}
     \caption{The isomorphic image in $\widetilde{S}_h(t)$}
     \label{ife_fun_2} 
\end{subfigure}
~
 \begin{subfigure}{.32\textwidth}
     \includegraphics[width=1.8in]{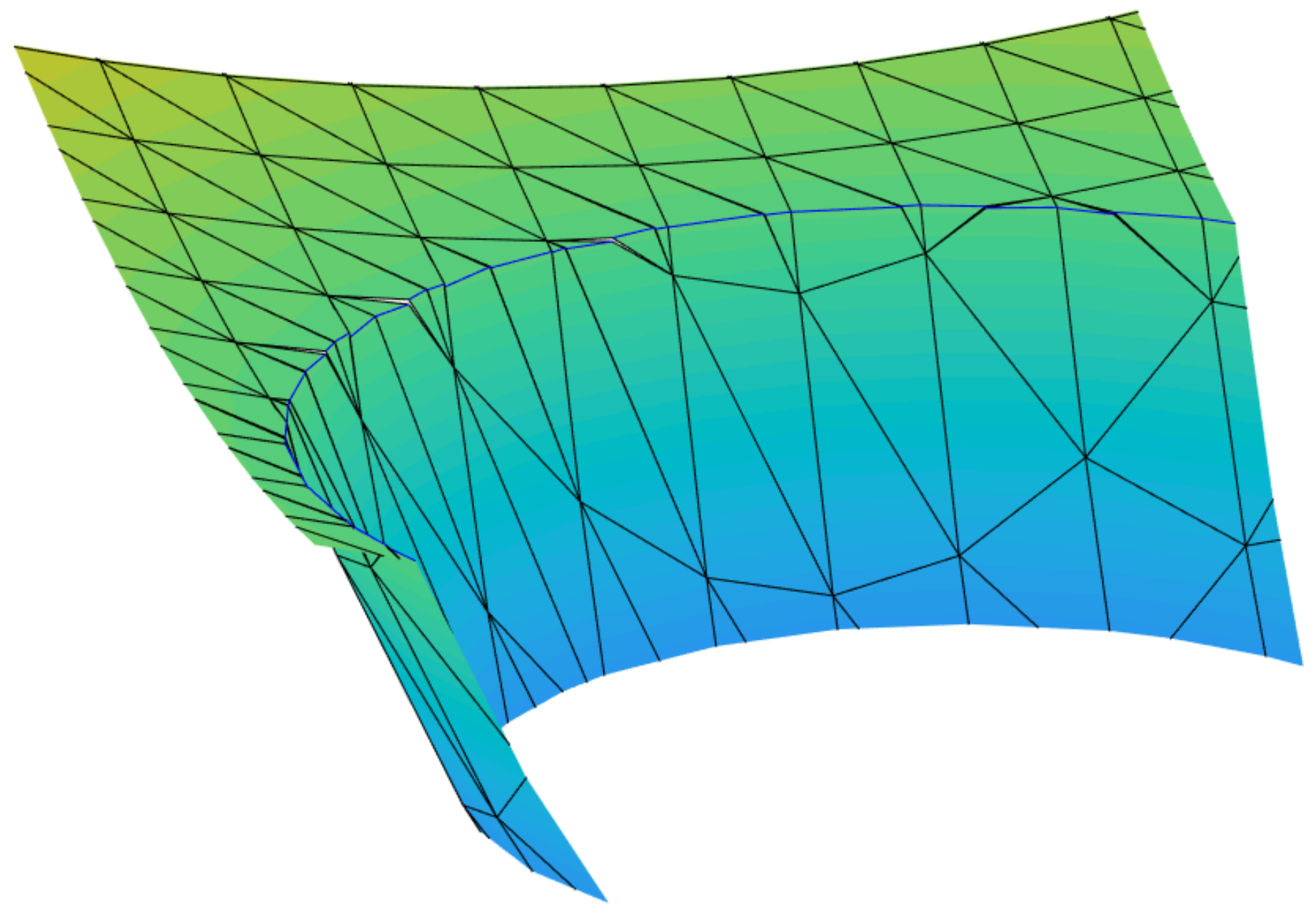}
     \caption{discontinuities on interface edges}
     \label{ife_fun_1_zoom} 
\end{subfigure}
     \caption{Plots of IFE functions}
  \label{fig:ife_fun} 
\end{figure}

%
%
However since the IFE functions loss the global continuity, the simple continuous Galerkin scheme yields the suboptimal convergence \cite{2015LinLinZhang}. To address this issue, the authors in \cite{2015LinLinZhang} added interior penalties on edges to handle the discontinuities. To describe the scheme, we define a symmetric bilinear form $a_h(\cdot,\cdot): W_h\times W_h\longrightarrow \mathbb{R}$ such that
\begin{equation}
\label{bilinear_form_1}
a_h(v_h,w_h) := \int_{\Omega} \beta \nabla v_h\cdot\nabla w_h dX - \sum_{e\in\mathring{\mathcal{E}}_h} \int_e \{ \beta \nabla v_h \cdot \bfn \}_e [w_h]_e ds - \sum_{e\in\mathring{\mathcal{E}}_h} \int_e \{ \beta \nabla w_h\cdot \bfn \}_e [v_h]_e ds + \sum_{e\in\mathring{\mathcal{E}}_h} \frac{\sigma_0}{|e|} \int_e [v_h]_e [w_h]_e ds
\end{equation}
where $\sigma_0=\sigma \tau^{-1}$ is the stability parameter large enough with $\tau$ being the step size specified later, and
\begin{equation}
\label{bilinear_form_2}
\{ \beta \nabla v\cdot \bfn \}_e = \frac{1}{2}\left( \beta \nabla v|_{T_1} \cdot \bfn + \beta \nabla v|_{T_2} \cdot \bfn \right), ~~~~~~ [v]_e = v|_{T_1} - v|_{T_2}
\end{equation}
with $T_1$ and $T_2$ being the neighbor elements of $e\in \mathring{\mathcal{E}}_h$. Note that $\sigma_0=\sigma \tau^{-1}$ is so called the super-penalty also used in \cite{2013Zunino}. Then the semi-discrete IFE scheme to the parabolic interface problem \eqref{model} is to find $u_h(\cdot,t)\in S_h(t)$ at each $t$ such that
\begin{equation}
\label{semi_discrete}
(\partial_t u_h, v_h )_{L^2(\Omega)} + a_h(u_h,v_h) = (f, v_h)_{L^2(\Omega)}, ~~~ \forall v_h \in S_h(t).
\end{equation}
We note that \eqref{bilinear_form_1} shares the same format as the symmetric interior penalty discontinuous Galerkin method \cite{1982Arnold,2008Riviere}. But it is essentially not a discontinuous Galerkin method since the degrees of freedom of test and trial spaces (the IFE spaces) are as the same as the continuous piecewise linear FE spaces, i.e., the isomorphism. 
Furthermore we highlight that $a_h(\cdot,\cdot)$ only needs to operate on the IFE spaces $S_h(t)\subseteq W_h$, then all the penalties on non-interface edges vanish and only those on interface edges $\mathcal{E}^i_h(t)$ are non-zero due to the discontinuities. Namely $\forall v_h,w_h \in S_h(t)$ there holds
\begin{equation}
\begin{split}
\label{bilinear_form_3}
a_h(v_h,w_h;t) = a_h(v_h,w_h)
 = &  \int_{\Omega} \beta \nabla v_h\cdot\nabla w_h dX - \sum_{e\in\mathcal{E}^i_h(t)} \int_e \{ \beta \nabla v_h\cdot \bfn \}_e [w_h]_e ds \\
  -&  \sum_{e\in\mathcal{E}^i_h(t)} \int_e \{ \beta \nabla w_h\cdot \bfn \}_e [v_h]_e ds + \sum_{e\in\mathcal{E}^i_h(t)} \frac{\sigma}{|e|} \int_e [v_h]_e [w_h]_e ds
 \end{split}
\end{equation}
which exactly reduces to the bilinear form of the so called partially penalized IFE (PPIFE) method introduced in \cite{2018GuoLinZhuang,2015LinLinZhang} for the elliptic interface problems and \cite{2020LinZhuang} for the parabolic interface problem but with the stationary interface. Actually since the bilinear form $a_h$ is only used on the IFE spaces, essentially only \eqref{bilinear_form_3} is required in computation. But here we prefer \eqref{bilinear_form_1} in analysis since it is uniform throughout dynamics independent of interface location. It makes the proposed method distinguished from other unfitted mesh methods requiring penalties on the interface itself \cite{2015BurmanClaus,2017HuangWuXiao}. This very unique feature of IFE methods enables us to restrict the variation in the algorithm to only the approximation spaces which suggests the employment of the fundamental framework of time-dependent adaptive finite element method in \cite{1991ErikssonJohnsonI}.

Based on the bilinear form $a_h(\cdot,\cdot)$, let's introduce some useful operators for analysis. At each $t$, we define an elliptic projection $\mathcal{R}_h(t)$ and a discrete Laplace operator $\mathcal{L}_h(t)$ such that
\begin{equation}
\label{ellip_proj}
\mathcal{R}_h(t)~:~ W_h(\Omega) \longrightarrow S_h(t), ~~~~ \text{with} ~ a_h(\mathcal{R}_h(t)w_h,v_h) = a_h(w_h,v_h) ~~ \forall v_h \in S_h(t),
\end{equation}
\begin{equation}
\label{discre_proj}
\mathcal{L}_h(t)~:~ W_h(\Omega) \longrightarrow S_h(t), ~~~~ \text{with} ~ (\mathcal{L}_h(t)w_h,v_h)_{L^2(\Omega)} = a_h(w_h,v_h) ~~ \forall v_h \in S_h(t),
\end{equation}
where these two operators are time-dependent since their images are time-dependent. Note that $\mathcal{R}_h(t)$ is well defined since $a_h$ is equivalent to the one in \eqref{bilinear_form_3} which is coercive on the IFE spaces (Lemma 4.1 in \cite{2015LinLinZhang}). The elliptic projection has been widely used in the semi and fully discrete analysis of numerical methods for time-dependent PDEs \cite{1991ErikssonJohnsonI,2008Riviere}. Its IFE version in \eqref{ellip_proj} has also been used for parabolic interface problems \cite{2015LinYangZhang,2020LinZhuang} with a stationary interface where the IFE spaces do not evolve so the related elliptic projection stay unchanged. Using Theorem 4.6 in \cite{2018GuoLinZhuang} for stationary interface problems, we immediately have the following estimate.

\begin{theorem}
\label{thm_ellip_proj_err}
There exists a constant $C$ such that for every $u\in \widetilde{H}^2(\Omega)$ with some interface $\Gamma(t)$
\begin{equation}
\label{thm_ellip_proj_err_eq0}
\| u - \mathcal{R}_h(t) u \|_{L^2(\Omega)} + h| u - \mathcal{R}_h(t) u |_{H^1(\Omega)} \le Ch^2 \| u \|_{H^2(\Omega)}.
\end{equation}
\end{theorem}
Some more delicate results about these operators and the bilinear form $a_h(\cdot,\cdot)$ will be derived in Section \ref{sec:pre_est}.


\subsection{Temporal Discretization}

In this subsection, we present a backward Euler time stepping IFE method for the parabolic moving interface model. As usual, we partition $[0,T]$ into $0=t_0<t_1<t_2<\cdots<t_{N}=T$ and define subintervals $J_n=(t_{n-1},t_n]$, $n=1,2,...,N$ which have equal length $\tau=| J_n |$. From now on, for simplicity at these discrete time points we shall denote the interpolation $\mathcal{I}^n_h=\mathcal{I}_h(t_n)$, the elliptic projection $\mathcal{R}^n_h=\mathcal{R}_h(t_n)$ and the discrete Laplace operator $\mathcal{L}^n_h=\mathcal{L}_h(t_n)$ as well as the IFE spaces $S^n_h=S_h(t_n)$, $n=0,1,\cdots,N$. In addition, for each sequence $v^n_h\in S^n_h$, $n=0,1,...,N$, we define the temporal finite difference operator
\begin{equation}
\label{time_fin_diff}
\delta_t v^n_h = \frac{v^n_h - v^{n-1}_h}{\tau}, ~~~ n=1,2,...N.
\end{equation}
Then the proposed backward Euler IFE method is to find a sequence $u^n_h\in S^n_h$ to approximate $u^n:=u(t_n)$, such that
\begin{equation}
\label{time_IFE_1}
(\delta_t u^n_h, v^n_h)_{L^2(\Omega)} + a_h(u^n_h, v^n_h) = (f(t_n),v^n_h)_{L^2(\Omega)}, ~~~~ \forall v^n_h \in S^n_h, ~~ n=1,\cdots,N,
\end{equation}
with $u^0_h = \mathcal{R}^0_h u_0$. Here we emphasize $a_h(\cdot,\cdot)$ can be understood as the one in \eqref{bilinear_form_3} and only the approximation spaces are changing in \eqref{time_IFE_1}. We note that \eqref{time_IFE_1} is readily used for computation, however it is not convenient for analysis in the present situation that approximation spaces are evolving at each step. To see this, let's apply the standard strategy by decomposing the error $u-u_h = u^n - u^n_h$ at $t=t_n$ into 
\begin{equation}
\label{err_decomp}
\xi_h^n = \mathcal{R}^n_hu - u^n_h \in S^n_h ~~~~~~ \text{and} ~~~~~~ \eta_h^n = u^n - \mathcal{R}^n_h u \in W_h.
\end{equation}
Then subtracting \eqref{time_IFE_1} from the counterpart for the exact solution $u$ and taking $v^n_h = \xi^n_h$, we obtain
\begin{equation}
\label{time_IFE_2}
( \delta_t \xi^n_h, \xi^n_h )_{L^2(\Omega)} + a_h( \xi^n_h, \xi^n_h)  = - (\delta_t \eta^n_h, \xi^n_h)_{L^2(\Omega)} - (\partial_t u^n - \delta_t u^n_h, \xi^n_h )_{L^2(\Omega)}.
\end{equation}
The key of the standard approach is to estimate each term in the right hand side of \eqref{time_IFE_2}. If the interface is stationary, i.e., $\mathcal{V}=0$, then $\mathcal{R}^n_h=\mathcal{R}_h$ is independent of time, and thus we obtain
$
\delta_t \eta^n_h = \delta_t u^n - \delta_t\mathcal{R}^n_h u =  \delta_t u^n - \mathcal{R}_h \delta_t u^n = (\mathcal{I}-\mathcal{R}_h) \delta_t u^n
$
where $\mathcal{I}$ is the identity operator. So the estimate directly follows from the approximation result of the elliptic projection \eqref{thm_ellip_proj_err_eq0}. However if the interface evolves, $\mathcal{R}^n_h$ is not commuting with $\delta_t$ anymore, and one can only obtain suboptimal estimate for $\delta_t \eta^n_h$. This issue is also discussed in Remark 3.1 of \cite{2013LehrenfeldReusken}.

\begin{remark}
\label{rem_time_IFE_1}
It is also interesting to note that the continuous temporal differential operator $\partial_t$ is not commuting with $\mathcal{R}_h(t)$ either. Actually for $u\in \widetilde{H}^2(\Omega)$ with some interface $\Gamma(t)$ it is easy to see $\mathcal{R}_h(t)\partial_t u\in S_h(t)$ but $\partial_t \mathcal{R}_h(t)u(t)\notin S_h(t)$ since the latter one does not satisfy the homogeneous jump conditions on $\Gamma_h(t)$ anymore, see \cite{J.Sokolowski_J.-P.Zolesio_1992} for more details. 
\end{remark}

Since the key issue is the variation of approximation spaces, it is reasonable to reconsider the fully discrete scheme \eqref{time_IFE_1} from the point of the view of the discontinuous Galerkin time stepping method introduced in \cite{1985ErikssonJohnsonThomee}, and this idea is then used for the time-dependent adaptive methods in \cite{1991ErikssonJohnsonI,1995ErikssonClaes}. To adopt this framework, we introduce the spaces
\begin{equation}
\label{IFE_space_time_new}
\mathbb{W}_h = \{ V_h\in L^2(0,T;W_h)~:~ V_h|_{J_n} \in H^1(J_n;W_h), ~ n=1,2,...N \},
\end{equation}
\begin{equation}
\label{IFE_space_time}
\mathbb{S}_h = \{ V_h\in L^2(0,T;W_h)~:~ V_h|_{J_n} := V^n_h\in S^n_h, ~ n=1,2,...N \},
\end{equation}
where functions in $\mathbb{S}_h$ are constant with respect to time on each interval $J_n$. In the error analysis, we mainly focus on the spaces in \eqref{IFE_space_time_new} and \eqref{IFE_space_time}, and use the capital notations, such as $V_h$, to refer functions in these spaces. For each $V_h\in \mathbb{W}_h$, we define $V^n_h=V_h(t^-_{n})$, $n=0,1,...,N$, and in particular if $V_h\in \mathbb{S}_h$ we have $V^n_h=V_h(t^-_{n})=V_h(t^+_{n-1})$. Furthermore, we denote $[[V_h]]_{n-1} :=V_h(t^+_{n-1}) - V_h(t^-_{n-1})$ as the jump at $t_{n-1}$, and in particular if $V_h\in \mathbb{S}_h$ we have $[[V_h]]_{n-1}=V^n_h-V^{n-1}_h$. Then the scheme \eqref{time_IFE_1} can be equivalently rewritten as finding $U_h\in\mathbb{S}_h$ such that for all $n$
\begin{equation}
\label{time_IFE_4}
( [[U_h]]_{n-1}, V^n_h )_{L^2(\Omega)} + \tau a_h(U^n_h,V^n_h) = \tau ( f(t_n), V^n_h )_{L^2(\Omega)}, ~~~~ \forall V_h\in \mathbb{S}_h.
\end{equation}
Note that we indeed have $U^n_h=u^n_h$ where $u^n_h$ are from the scheme \eqref{time_IFE_1}, and in the following discussion we shall focus on $U^n_h$ to avoid confusion of notations. Then summing \eqref{time_IFE_4} from $n=1$ to $N$, we have the equivalent time stepping scheme involving time integration: find $U_h\in \mathbb{S}_h$ such that
\begin{equation}
\label{time_IFE_5}
A_h(U_h,V_h) = F_h(V_h), ~~~~ \forall V_h \in \mathbb{S}_h
\end{equation}
where the bilinear form $A_h(\cdot,\cdot):\mathbb{W}_h\times \mathbb{W}_h \rightarrow \mathbb{R}$ is defined as
\begin{equation}
\begin{split}
\label{time_IFE_6}
A_h(U_h,V_h) &= \sum_{n=1}^N \int_{J_n} (\partial_t U_h, V_h)_{L^2(\Omega)} dt + \tau \sum_{n=1}^N a_h(U_h(t^-_n),V_h(t^-_n)) \\
&  + \sum_{n=2}^N ( [[U_h]]_{n-1}, V_h(t^-_n) )_{L^2(\Omega)} + (U_h(t^+_0), V_h(t^-_1))_{L^2(\Omega)}
\end{split}
\end{equation}
where the term $(\partial_t U_h, V_h)_{L^2(\Omega)}$ is needed due to $\partial_t u$ of the original equation, and the linear functional $F_h:\mathbb{W}_h \rightarrow \mathbb{R}$ is
\begin{equation}
\label{time_IFE_7}
F_h(V_h) = \tau \sum_{n=1}^N ( f(t_n), V_h(t^-_n) )_{L^2(\Omega)}  + (U_h(t^-_0),V_h(t^-_1))_{L^2(\Omega)}
\end{equation}
where $U_h(t^-_0)=U^0_h = R^0_hu_0$ is the given initial condition. We emphasize that the bilinear form $A_h$ and the linear form $F_h$ are defined for time-dependent functions not for sequences. Although \eqref{time_IFE_5} is essentially equivalent to \eqref{time_IFE_1}, \eqref{time_IFE_5} is more suitable for analysis. 


\section{Some Fundamental Estimates}
\label{sec:pre_est}

In this section, we prepare some fundamental estimates which will be used for stability and error analysis later. Although the IFE spaces $S_h(t)$ are not $H^1$ conforming globally, they are locally $H^1$ functions on each element. This feature enables us to show some nice properties. First of all, we show the following Poincar\'e-Friedrichs-type inequality.
\begin{lemma}
\label{lem_pf_inequa}
There exists a constant $C$ such that for each element $T$
\begin{equation}
\label{lem_pf_inequa_eq0}
\min_{\chi\in\mathbb{P}_0(T)} \| v_h - \chi \|_{L^2(T)} \le Ch_T | v_h |_{H^1(T)}, ~~~ \forall v_h \in S_{h,T}(t).
\end{equation}
\end{lemma}
\begin{proof}
On non-interface elements, the result is trivial since $S_{h,T}(t)=\mathbb{P}_1(T)$. The result on interface elements also directly follows from the fact that $\mathbb{P}_0(T)\subset S_{h,T}(t)\subset H^1(T)$.
\end{proof}
Recalling that each interpolation $\mathcal{I}_h(t)$ is an isomorphism from $\widetilde{S}_h$ to $S_h(t)$, in the following several results for convenience of presentation, we focus on its inverse $\mathcal{I}^{-1}_h(t)=:\tilde{\mathcal{I}}_h(t):S_h(t)\rightarrow \widetilde{S}_h$. We then show some stability estimatse. 
\begin{lemma}
\label{lem_stability}
There exist constants $c$ and $C$ such that for each element $T$
\begin{equation}
\label{lem_stability_eq0}
c| v_h |_{H^j(T)} \le | \tilde{\mathcal{I}}_{h,T}(t) v_h |_{H^j(T)} \le C | v_h |_{H^j(T)}, ~ j=0,1, ~~~ \forall v_h\in S_{h,T}(t). 
\end{equation}
\end{lemma}
\begin{proof}
Again the results on non-interface elements are trivial since the isomorphism reduces to the identity operator, and we only discuss the interface elements. We first show the estimate for $j=0$. On each element $T$ with the vertices $A_i$ and edges $e_i$, let $\psi_{i,T}$ and $\phi_{i,T}$, $i=1,2,3$, be the Lagrange shape functions of the IFE space $S_{h,T}(t)$ and the FE space $\widetilde{S}_{h,T}$, respectively. For each $v_h\in S_{h,T}(t)$, noting $v_h|_{\partial T}\in H^1(\partial T)$, we let $e$ be one neighbor edge of $A_1$ and obtain
\begin{equation}
\begin{split}
\label{lem_stability_eq1}
v_h(A_1) & \le Ch^{-1/2}_T \| v_h \|_{L^2(e)} + Ch^{1/2}_T \| \partial_{\bft}v_h \|_{L^2(e)} \\
& \le Ch^{-1}_T \| v_h \|_{L^2(T)} + C \| \nabla v_h \|_{L^2(T)} \le Ch^{-1}_T \| v_h \|_{L^2(T)}
\end{split}
\end{equation}
where we have used the standard trace inequality from $A_1$ to $e$ in the first inequality, the trace inequality for IFE functions given by \eqref{trace_inequa} in the second inequality and the inverse inequality for IFE functions given by \eqref{inver_inequa} in the third inequality. Then by the boundedness of $\phi_{i,T}$ and \eqref{lem_stability_eq1} we have
\begin{equation}
\label{lem_stability_eq2}
\vertii{ \tilde{\mathcal{I}}_h(t) v_h}_{L^2(T)} = \vertii{ \sum_{i=1}^3 v_h(A_i) \phi_{i,T} }_{L^2(T)} \le \| v_h \|_{L^2(T)}
\end{equation}
which gives the right inequality of \eqref{lem_stability_eq0} for $j=0$. The left inequality can be shown through a similar argument with the help of the boundedness of $\psi_{i,T}$ in \eqref{boundedness}.

As for $j=1$, we note that $\mathbb{P}_0(T)\subseteq S_{h,T}\cap \widetilde{S}_{h,T}$ and thus $\tilde{\mathcal{I}}_{h,T}(t)$ preserves constants. Then for every $\chi\in \mathbb{P}_0(T)$ and $v_h \in S_{h,T}(t)$ we obtain form the inverse inequality \eqref{inver_inequa}, the $L^2$ stability above and Lemma \ref{lem_pf_inequa} that
\begin{equation}
\begin{split}
\label{lem_stability_eq3}
|\tilde{\mathcal{I}}_{h,T}(t)v_h |_{H^1(T)} & = | \tilde{\mathcal{I}}_{h,T}(t)v_h - \chi |_{H^1(T)} \le Ch^{-1}_T \| \tilde{\mathcal{I}}_{h,T}(t)v_h - \chi \|_{L^2(T)} \\
& =  Ch^{-1}_T \| \tilde{\mathcal{I}}_{h,T}(t) (v_h - \chi) \|_{L^2(T)} \le  Ch^{-1}_T \| v_h - \chi \|_{L^2(T)} \le C | v_h |_{H^1(T)}
\end{split}
\end{equation}
which gives the right inequality of \eqref{lem_stability_eq0} for $j=1$. The left one can be proved by a similar argument.
\end{proof}

\begin{lemma}
\label{lem_glob_stability}
There exist constants $c$ and $C$ such that
\begin{equation}
\label{lem_glob_stability_eq0}
c| v_h |_{H^j(\Omega)} \le | \tilde{\mathcal{I}}_{h}(t) v_h |_{H^j(\Omega)} \le C | v_h |_{H^j(\Omega)}, ~ j=0,1, ~~~ \forall v_h\in S_{h}(t). 
\end{equation}
\end{lemma}
\begin{proof}
It immediately follows from the local stability in Lemma \ref{lem_stability}.
\end{proof}


\begin{lemma}
\label{thm_interp_err}
There exists a constant $C$ such that for each element $T$
\begin{equation}
\label{thm_interp_err_eq0}
\| v_h - \tilde{\mathcal{I}}_{h,T}(t)v_h \|_{L^2(T)} \le Ch_T | v_h|_{H^1(T)} ~~~~ \forall v_h\in S_{h,T}(t).
\end{equation}
\end{lemma}
\begin{proof}
Again the results are trivial on non-interface elements. On each interface element $T$, using the properties again that $\mathbb{P}_0(T)\subseteq S_{h,T}\cap \widetilde{S}_{h,T}$ and $\tilde{\mathcal{I}}_{h,T}(t)$ preserves constants, and the $L^2$ stability in Lemma \ref{lem_stability}, we have for any constant $\chi\in \mathbb{P}_0(T)$
\begin{equation}
\begin{split}
\label{thm_interp_err_eq1}
\| v_h - \tilde{\mathcal{I}}_{h,T}(t)v_h \|_{L^2(T)} &= \| v_h - \chi + \tilde{\mathcal{I}}_{h,T}(t)\chi - \tilde{\mathcal{I}}_{h,T}(t)v_h \|_{L^2(T)} \le C \| v_h - \chi \|_{L^2(T)}.
\end{split}
\end{equation}
Then the estimate in Lemma \ref{lem_pf_inequa} yields the desired result.
\end{proof}

The stability in Lemma \ref{lem_glob_stability} and the estimate in Lemma \ref{thm_interp_err} enable us to show that the IFE functions have similar properties as FE functions such as the following global trace inequality.

\begin{theorem}
\label{thm_glob_trace}
There exists a constant $C$ such that
\begin{equation}
\label{thm_glob_trace_eq0}
\| v_h \|_{L^2(\Gamma(t))} \le C \| v_h \|_{H^1(\Omega)} ~~~ \forall v_h \in S_h(t).
\end{equation}
\end{theorem}
\begin{proof}
Note that \eqref{thm_glob_trace_eq0} is true for FE functions due to the global $H^1$-conformity. Given each $v_h\in S_h(t)$, by the triangular inequality we have
\begin{equation}
\label{thm_glob_trace_eq1}
\| v_h \|_{L^2(\Gamma(t))} \le \| v_h - \tilde{\mathcal{I}}_hv_h \|_{L^2(\Gamma(t))} + \| \tilde{\mathcal{I}}_hv_h \|_{L^2(\Gamma(t))}.
\end{equation}
Since $\tilde{\mathcal{I}}_hv_h \in \widetilde{S}_h(t)$, by the trace inequality and Lemma \ref{lem_glob_stability} we have
\begin{equation}
\label{thm_glob_trace_eq2}
\| \tilde{\mathcal{I}}_hv_h \|_{L^2(\Gamma(t))} \le C\| \tilde{\mathcal{I}}_h v_h \|_{H^1(\Omega)} \le C\| v_h \|_{H^1(\Omega)}. 
\end{equation}
It remains to estimate the first term in the right hand side of \eqref{thm_glob_trace_eq1}. On each interface element $T$, by the trace inequality Lemma 3.1 in \cite{2016WangXiaoXu}, Lemma \ref{thm_interp_err} and Lemma \ref{lem_glob_stability} with $j=1$ we have
\begin{equation}
\begin{split}
\label{thm_glob_trace_eq3}
\| v_h - \tilde{\mathcal{I}}_hv_h \|_{L^2(\Gamma(t)\cap T)} & \le Ch_T^{-1/2} \| v_h - \tilde{\mathcal{I}}_hv_h \|_{L^2(T)} + Ch_T^{1/2} | v_h - \tilde{\mathcal{I}}_hv_h |_{H^1(T)}  \le C h^{1/2}_T | v_h |_{H^1(T)}.
\end{split}
\end{equation}
Therefore, by the assumption $h_T\le C$ there holds
\begin{equation}
\label{thm_glob_trace_eq4}
\| v_h - \tilde{\mathcal{I}}_hv_h \|^2_{L^2(\Gamma(t))} = \sum_{T\in\mathcal{E}^i_h(t)} \| v_h - \tilde{\mathcal{I}}_hv_h \|^2_{L^2(\Gamma(t)\cap T)} \le C \sum_{T\in\mathcal{E}^i_h(t)}  |v_h|^2_{H^1(T)} \le C | v_h |^2_{H^1(\Omega)}.
\end{equation}
Putting \eqref{thm_glob_trace_eq2} and \eqref{thm_glob_trace_eq4} into \eqref{thm_glob_trace_eq1} yields the desired result.
\end{proof}
We emphasize that the results above are basically some delicate estimates on the difference between the IFE functions and their FE isomorphic images. To our best knowledge, these results have not appeared in the literature.

Next we proceed to discuss the coercivity of the bilinear form $a_h(\cdot,\cdot)$ in \eqref{bilinear_form_1}. Due to the equivalence in \eqref{bilinear_form_3}, the coercivity is already given in Lemma 4.1 in \cite{2015LinLinZhang} and Theorem 4.3 in \cite{2018GuoLinZhuang}. But in order to handle the dynamical IFE spaces, we need more delicate results. For this purpose, we first introduce the uniform energy norm on the general broken space $W_h$:
\begin{equation}
\begin{split}
\label{norm_1}
\vertiii{v_h}^2_{h} := \| \sqrt{\beta}\nabla v_h \|^2_{L^2(\Omega)} + \sigma h^{-1} \tau^{-1} \sum_{e\in\mathring{\mathcal{E}}_h} \| [v]_e \|^2_{L^2(e)} + h\tau\sigma^{-1} \sum_{e\in\mathring{\mathcal{E}}_h} \| \{ \beta \nabla v\cdot\bfn \}_e \|^2_{L^2(e)} .
\end{split}
\end{equation}
It is easy to see this is a norm on the broken space $W_h$, and we note that this energy norm is widely used in the interior penalty discontinuous Galerkin methods \cite{1982Arnold}. Here we use it for the IFE spaces and note that all the terms $\| [v] \|^2_{L^2(e)}$ on non-interface edges just vanish. A similar energy norm is also used in \cite{2018GuoLinZhuang} with only the penalty terms on interface edges. Here we add the penalty terms on all the edges such that the norm format also keeps unchanged in the dynamics. This feature is important for the following error analysis. Using \eqref{thm_appro_eq0} and the argument of Theorem 4.5 in \cite{2018GuoLinZhuang}, we can show the following estimate.
\begin{theorem}
\label{thm_appro_energy}
Suppose $u\in \widetilde{H}^2(\Omega)$ with some interface $\Gamma(t)$, there exists a constant $C$ such that
\begin{equation}
\label{thm_appro_energy_eq0}
\vertiii{ u - \mathcal{I}_h(t)u }_h \le Ch \vertii{ u }_{H^2(\Omega)}.
\end{equation}
\end{theorem}
We also have the following coercivity in terms of the energy norm $\vertiii{\cdot}_h$.
\begin{theorem}
\label{thm_coer}
Suppose $\sigma$ is sufficiently larger, then there exists constants $\kappa_1$ and $\kappa_2$ such that
\begin{subequations}
\label{thm_coer_eq0}
\begin{align}
    & a_h(v_h,v_h) \ge \kappa_1 \vertiii{v_h}^2_h, ~~~~~~~~~~~ \forall v_h \in S_h(t), \label{thm_coer_eq01} \\
    & a_h(v_h,w_h) \le \kappa_2 \vertiii{v_h}_h \vertiii{w_h}_h, ~~~ \forall v_h,w_h \in W_h. \label{thm_coer_eq02}
\end{align}
\end{subequations}
\end{theorem}
\begin{proof}
The argument is almost as the same as Theorems 4.3 and 4.4 in \cite{2018GuoLinZhuang}.
\end{proof}
Since $a_h(\cdot,\cdot)$ is coercive, we can define another uniform norm directly induced from $a_h(\cdot,\cdot)$:
\begin{equation}
\label{norm_2}
\vertiii{v_h}^2_a = a_h(v_h,v_h), ~~~~ \forall v_h\in S_h(t),
\end{equation}
which, again, is independent of $\Gamma_h(t)$ during the dynamics. The two inequalities in Theorem \ref{thm_coer} together show that $\vertiii{\cdot}_a$ is equivalent to $\vertiii{\cdot}_h$ on each $S_h(t)$ where the hidden constant is uniformly bounded. However we emphasize that $a_h(\cdot,\cdot)$ is not coercive on general broken Sobolev spaces such as $W_h$. So in addition to \eqref{thm_coer_eq01}, we also need the following weak coercivity of which the underling idea is also used in \cite{2006ChrysafinosWalkington,2007FengKarakashian} for error analysis of DG methods on dynamic meshes.
\begin{theorem}
\label{thm_weak_coer}
Suppose $\sigma$ is sufficiently large, then there exist constants $\delta_0$ and $\delta_1$ such that
\begin{equation}
\label{thm_weak_coer_eq0}
a_h(v_h,v_h) \ge \delta_0 \vertiii{v_h}^2_h - \delta_1 h \tau \sigma \sum_{e\in\mathring{\mathcal{E}}_h} \vertii{ \{\beta \nabla v_h\cdot\bfn\}_e }^2_{L^2(e)}, ~~~~~ \forall v_h\in W_h. 
\end{equation}
\end{theorem}
\begin{proof}
Let $\delta_0$ be a constant to be determined later. Then the Young's inequality yields
\begin{equation}
\begin{split}
\label{thm_weak_coer_eq1}
a_h(v_h,v_h) - \delta_0 \vertiii{v_h}^2_h &= (1-\delta_0) \| \sqrt{\beta} \nabla v_h \|^2_{L^2(\Omega)} + (1-\delta_0) \sigma h^{-1} \tau^{-1} \sum_{e\in\mathring{\mathcal{E}}_h} \| [v_h]_e \|^2_{L^2(e)} \\
&- \sum_{e\in\mathring{\mathcal{E}}_h}  \{ \beta \nabla v_h \cdot \bfn \}_e [ v_h ]_e ds - \delta_0  \sigma^{-1} h \tau \sum_{e\in\mathring{\mathcal{E}}_h} \vertii{ \{\beta \nabla v_h\cdot\bfn\}_e }^2_{L^2(e)} \\
& \ge (1-\delta_0) \| \sqrt{\beta} \nabla v_h \|^2_{L^2(\Omega)} +  (1-\delta_0 - \epsilon) \sigma h^{-1} \tau^{-1} \sum_{e\in\mathring{\mathcal{E}}_h} \| [v_h]_e \|^2_{L^2(e)} \\
-& \left( \delta_0 + \frac{1}{4 \epsilon} \right) \sigma^{-1} h \tau  \sum_{e\in\mathring{\mathcal{E}}_h} \vertii{ \{\beta \nabla v_h\cdot\bfn\}_e }^2_{L^2(e)}.
\end{split} 
\end{equation}
Taking $\delta_0=\epsilon=1/4$ and $\delta_1=\delta_0+1/(4\epsilon)=5/4$ finishes the proof.
\end{proof}

In the discussion below, we always assume $\sigma$ is sufficiently large such that the coercivity above hold without explicitly mentioning again. Then we present a discrete Poincar\'e inequality and the stability of the elliptic projection $\mathcal{R}_h(t)$.
\begin{theorem}
\label{thm_dis_poinc}
There exists a constant $C$ such that for each $v_h\in W_h$
\begin{equation}
\label{thm_dis_poinc_eq0}
\| v_h \|_{L^2(\Omega)} \le C \vertiii{ v_h }_h.
\end{equation}
\end{theorem}
\begin{proof}
The proof is in the same spirit of Lemma 2.1 in \cite{1982Arnold}.
\end{proof}
\begin{theorem}
\label{thm_lap_stab}
There exists a constant $C$ such that $ \vertiii{ \mathcal{R}_h(t) v_h }_h \le C \vertiii{ v_h }_h$.
\end{theorem}
\begin{proof} 
It immediately follows from Theorem \ref{thm_coer}. 
\end{proof}

Now let's go back to the operators $\mathcal{R}^n_h$ and $\mathcal{L}^n_h$ at $t_n$, $n=0,1,...,N$, and show the following estimates.

\begin{theorem}
\label{thm_Rn_h_err}
There exists a constant $C$ such that
\begin{equation}
\label{thm_Rn_h_err_eq0}
\vertiii{ \mathcal{R}^n_h v_{h} - v_{h} }_h \le Ch \| \mathcal{L}^{n-1}_h v_h \|_{L^2(\Omega)} + C \tau^{1/2} \vertiii{ v_h }_{h} , ~~~~ \forall v_h\in S^{n-1}_h, ~~ n=1,...,N.
\end{equation}
\end{theorem}
\begin{proof}
Using Theorem \ref{thm_weak_coer}, for each $v_h\in S^{n-1}_h$ we have
\begin{equation}
\begin{split}
\label{thm_Rn_err_eq4}
 \vertiii{ \mathcal{R}^n_h v_{h} - v_{h} }_h & \le C  \verti{ a_h(\mathcal{R}^n_h v_{h} - v_{h}, \mathcal{R}^n_h v_{h} - v_{h}) }^{1/2} \\
&+ C h^{1/2} \tau^{1/2} \sigma^{1/2} \sum_{e\in\mathring{\mathcal{E}}_h} \vertii{ \{\beta \nabla ( \mathcal{R}^n_h v_{h} - v_{h} ) \cdot\bfn\}_e }_{L^2(e)}
\end{split}
\end{equation}
where we denote $\chi_1=\verti{ a_h( \mathcal{R}^n_h v_{h} - v_{h}, \mathcal{R}^n_h v_{h} - v_{h}) }^{1/2}$ and $\chi_2 = \sum_{e\in\mathring{\mathcal{E}}_h} \vertii{ \{\beta \nabla ( \mathcal{R}^n_h v_{h} - v_{h} ) \cdot\bfn\}_e }_{L^2(e)}$, respectively, for the right hand side. For $\chi_1$, inspired by argument of Lemma 2.2 in \cite{1991ErikssonJohnsonI} we obtain:
\begin{equation}
\begin{split}
\label{thm_Rn_err_eq5}
\chi^2_1 = \verti{ a_h( \mathcal{R}^n_h v_{h} - v_{h}, \mathcal{R}^n_h v_{h} - v_{h}) } & = \verti{ a_h( \mathcal{R}^n_h v_{h} - v_{h}, v_{h}) } = \verti{ a_h( \mathcal{R}^{n-1}_h \mathcal{R}^n_h v_{h} - v_{h}, v_{h}) }
\end{split}
\end{equation}
where we have used the definition of $\mathcal{R}^n_h$ and $\mathcal{R}^{n-1}_h$ and the fact $v_h\in S^{n-1}_h$. Then since $\mathcal{R}^{n-1}_h \mathcal{R}^n_h v_{h}\in S^{n-1}_h$, using the discrete Laplace operator \eqref{discre_proj}, we have
\begin{equation}
\label{thm_Rn_err_eq6}
\verti{ a_h( \mathcal{R}^{n-1}_h \mathcal{R}^n_h v_{h} - v_{h}, v_{h}) } = |( \mathcal{R}^{n-1}_h \mathcal{R}^n_h v_{h} - v_{h}, \mathcal{L}^{n-1}_h v_h)_{L^2(\Omega)}| \le \| \mathcal{R}^{n-1}_h \mathcal{R}^n_h v_{h} - v_{h} \|_{L^2(\Omega)} \| \mathcal{L}^{n-1}_h v_h \|_{L^2(\Omega)}.
\end{equation}
We need to estimate the first term in the right hand side of \eqref{thm_Rn_err_eq6}. For this purpose, we split this term into
\begin{equation}
\label{thm_Rn_err_eq6_1}
\mathcal{R}^{n-1}_h \mathcal{R}^n_h v_{h} - v_{h} = \left( \mathcal{R}^{n-1}_h - \mathcal{I} \right) \left( \mathcal{R}^{n}_h - \mathcal{I} \right)v_h + \left( \mathcal{R}^{n}_h - \mathcal{I} \right)v_h.
\end{equation}
By the duality argument, we define two auxiliary functions $z_1\in \widetilde{H}^2_0(\Omega)$ and $z_2\in \widetilde{H}^2_0(\Omega)$ with the interface $\Gamma(t_{n-1})$ and $\Gamma(t_{n})$, respectively, satisfying the equations
\begin{equation}
\begin{split}
\label{thm_Rn_err_eq6_2}
& \nabla\cdot(\beta \nabla z_1) =  \left( \mathcal{R}^{n-1}_h - \mathcal{I} \right) \left( \mathcal{R}^{n}_h - \mathcal{I} \right)v_h,\\
& \nabla\cdot(\beta \nabla z_2) =  \left( \mathcal{R}^{n}_h - \mathcal{I} \right)v_h.
\end{split}
\end{equation}
For first equation in \eqref{thm_Rn_err_eq6_2}, multiplying it by $ \left( \mathcal{R}^{n-1}_h - \mathcal{I} \right) \left( \mathcal{R}^{n}_h - \mathcal{I} \right)v_h$ and noticing the penalties of $a_h$ are added on every interior edge, we use integration by parts to obtain
\begin{equation}
\begin{split}
\label{thm_Rn_err_eq6_3}
\vertii{ \left( \mathcal{R}^{n-1}_h - \mathcal{I} \right) \left( \mathcal{R}^{n}_h - \mathcal{I} \right)v_h }^2_{L^2(\Omega)} & = a_h( z_1, \left( \mathcal{R}^{n-1}_h - \mathcal{I} \right) \left( \mathcal{R}^{n}_h - \mathcal{I} \right)v_h ) \\
& = a_h( z_1 - \mathcal{I}_h^{n-1} z_1, \left( \mathcal{R}^{n-1}_h - \mathcal{I} \right) \left( \mathcal{R}^{n}_h - \mathcal{I} \right)v_h ) \\
& \le \vertiii{ z_1 - \mathcal{I}_h^{n-1} z_1 }_h \vertiii{ \left( \mathcal{R}^{n-1}_h - \mathcal{I} \right) \left( \mathcal{R}^{n}_h - \mathcal{I} \right)v_h }_h
\end{split}
\end{equation}
where $\mathcal{I}_h^{n-1}$ is the interpolation to $S_h^{n-1}$ for the interface $\Gamma(t_{n-1})$. Then Theorem \ref{thm_appro_energy} and regularity of elliptic interface problems \cite{1998ChenZou,2002HuangZou} yield
\begin{equation}
\label{thm_Rn_err_eq6_4}
\vertiii{ z_1 - \mathcal{I}_h^{n-1} z_1 }_h \le Ch \| z_1 \|_{H^2(\Omega)} \le C h \vertii{ \left( \mathcal{R}^{n-1}_h - \mathcal{I} \right) \left( \mathcal{R}^{n}_h - \mathcal{I} \right)v_h }_{L^2(\Omega)}.
\end{equation}
Putting \eqref{thm_Rn_err_eq6_4} into \eqref{thm_Rn_err_eq6_3}, and using the stability of $\mathcal{R}^{n-1}_h$, we have
\begin{equation}
\label{thm_Rn_err_eq6_5}
\vertii{ \left( \mathcal{R}^{n-1}_h - \mathcal{I} \right) \left( \mathcal{R}^{n}_h - \mathcal{I} \right)v_h }_{L^2(\Omega)}
\le Ch \vertiii{ \left( \mathcal{R}^{n-1}_h - \mathcal{I} \right) \left( \mathcal{R}^{n}_h - \mathcal{I} \right)v_h }_h
\le Ch \vertiii{ \left( \mathcal{R}^{n}_h - \mathcal{I} \right)v_h }_h.
\end{equation}
Employing the second equation in \eqref{thm_Rn_err_eq6_2} with the help of the interpolation $\mathcal{I}_h(t_{n})$, we can similarly obtain
\begin{equation}
\label{thm_Rn_err_eq6_6}
\vertii{  \left( \mathcal{R}^{n}_h - \mathcal{I} \right)v_h }_{L^2(\Omega)} \le Ch \vertiii{ \left( \mathcal{R}^{n}_h - \mathcal{I} \right)v_h }_h.
\end{equation}
Combining \eqref{thm_Rn_err_eq6_5} and \eqref{thm_Rn_err_eq6_6} together and using \eqref{thm_Rn_err_eq4}, we have
\begin{equation}
\begin{split}
\label{thm_Rn_err_eq7}
 \| \mathcal{R}^{n-1}_h  \mathcal{R}^n_h v_{h} - v_{h} \|_{L^2(\Omega)}  \le C h \vertiii{ \mathcal{R}^n_h v_{h} - v_{h} }_h    \le Ch \chi_1 + Ch^{3/2}\tau^{1/2} \sigma^{1/2} \chi_2.
 \end{split}
\end{equation}
Now substituting \eqref{thm_Rn_err_eq6} and \eqref{thm_Rn_err_eq7} into \eqref{thm_Rn_err_eq5}, we obtain
\begin{equation}
\label{thm_Rn_err_eq8}
\chi^2_1 \le Ch \| \mathcal{L}^{n-1}_h v_h \|_{L^2(\Omega)} \chi_1 + Ch^{3/2}\tau^{1/2} \sigma^{1/2} \| \mathcal{L}^{n-1}_h v_h \|_{L^2(\Omega)}  \chi_2.
\end{equation}
Note that this quadratic inequality is certainly solvable. Solving this quadratic inequality, we have
\begin{equation}
\label{thm_Rn_err_eq9}
\chi_1 \le Ch  \| \mathcal{L}^{n-1}_h v_h \|_{L^2(\Omega)} + Ch^{3/4}\tau^{1/4} \sigma^{1/4} \| \mathcal{L}^{n-1}_h v_h \|^{1/2}_{L^2(\Omega)}  \chi^{1/2}_2. 
\end{equation}
Then the arithmetic inequality leads to
\begin{equation}
\label{thm_Rn_err_eq10}
Ch^{3/4}\tau^{1/4} \sigma^{1/4} \| \mathcal{L}^{n-1}_h v_h \|^{1/2}_{L^2(\Omega)}  \chi^{1/2}_2\le  Ch \| \mathcal{L}^{n-1}_h v_h \|_{L^2(\Omega)} + Ch^{1/2}\tau^{1/2} \sigma^{1/2} \chi_2 .
\end{equation}
Combining \eqref{thm_Rn_err_eq10} and \eqref{thm_Rn_err_eq9} we obtain
\begin{equation}
\label{thm_Rn_err_eq11}
 \chi_1 \le Ch \| \mathcal{L}^{n-1}_h v_h \|_{L^2(\Omega)} + Ch^{1/2}\tau^{1/2} \sigma^{1/2} \chi_2 
\end{equation}
which is then put into \eqref{thm_Rn_err_eq4}. Now it remains to estimate $\chi_2$ through the trace inequality \eqref{trace_inequa}:
\begin{equation}
\begin{split}
\label{thm_Rn_err_eq12}
\chi_2  & \le C\sum_{e\in\mathring{\mathcal{E}}_h} \left( \vertii{ \beta \nabla \mathcal{R}^n_h v_{h}  \cdot\bfn }_{L^2(e)} +  \vertii{ \beta \nabla v_{h}  \cdot\bfn }_{L^2(e)} \right) \\
& \le C \sum_{T\in \mathcal{T}_h} \left(  h^{-1/2}_T \vertii{ \beta \nabla \mathcal{R}^n_h v_{h}  }_{L^2(T)} + h^{-1/2}_T \vertii{ \beta \nabla  v_{h}  }_{L^2(T)}  \right) \\
& \le Ch^{-1/2} \left( \vertiii{ \mathcal{R}^n_h v_{h} }_h + \vertiii{  v_{h} }_h \right)  \le Ch^{-1/2}  \vertiii{  v_{h} }_h
\end{split}
\end{equation}
where in the last inequality we have also use the stability of $\mathcal{R}^n_h$ in terms of the norm $\vertiii{\cdot}_h$. Finally combining \eqref{thm_Rn_err_eq11}, \eqref{thm_Rn_err_eq12} and \eqref{thm_Rn_err_eq4}, we have the desired result. 
\end{proof}

\begin{theorem}
\label{thm_Rn_err}
There exists a constant $C$ such that
\begin{equation}
\label{thm_Rn_err_eq0}
\| \mathcal{R}^n_h v_{h} - v_{h} \|_{L^2(\Omega)} \le Ch^2 \| \mathcal{L}^{n-1}_h v_h \|_{L^2(\Omega)} + C \tau^{1/2} h \vertiii{ v_h }_{h} , ~~~~ \forall v_h\in S^{n-1}_h, ~~ n=1,...,N.
\end{equation}
\end{theorem}
\begin{proof}
By the duality argument similar to \eqref{thm_Rn_err_eq6_2}-\eqref{thm_Rn_err_eq6_5} above we can show
\begin{equation}
\begin{split}
\label{thm_Rn_err_eq3_1}
\| \mathcal{R}^n_h v_{h} - v_{h} \|_{L^2(\Omega)} & \le C h \vertiii{ \mathcal{R}^n_h v_{h} - v_{h} }_h 
\end{split}
\end{equation}
which gives the desired result by Theorem \ref{thm_Rn_h_err}.
\end{proof}

\begin{remark}
\label{rem_Rn_err}
Theorem \ref{thm_Rn_err} can be understood as a generalized estimate of (2.12) in \cite{1991ErikssonJohnsonI} that a standard continuous Galerkin method is used and the bilinear form is simply the standard $H^1$ inner product. A similar result is also derived in \cite{2013Zunino} for the $H^1$ inner product with discontinuous coefficients. But as the major difference/difficulty, the bilinear form $a_h(\cdot,\cdot)$ used in the IFE method may not be coercive on the general broken space $W_h$ which is the essential reason of the extra term $C\tau^{1/2}h\vertiii{v_h}_h$ appearing in \eqref{thm_Rn_err_eq0}. This feature also makes the proof much more technical.
\end{remark}


\section{Error Estimates}
\label{sec:error_est}

In this section, we proceed to estimate the errors of the fully discrete IFE scheme. For simplicity we shall assume $f=0$. We begin with the following stability results.

\begin{theorem}[Stability]
\label{thm_stab}
Given each initial condition $U^0_h$, let $U^n_h$, $n=1,2.,...,N$ be the solutions to the scheme \eqref{time_IFE_4} or \eqref{time_IFE_5}, then there exists a constant such that for any integer $M\le N$
\begin{subequations}
\label{thm_stab_eq0}
\begin{align}
    & \| U^M_h \|^2_{L^2(\Omega)} + 2\tau \kappa_1 \sum_{n=1}^M \vertiii{ U^n_h }^2_h + \sum_{n=1}^M \| [[ U_h ]]_{n-1} \|^2_{L^2(\Omega)} \le  \| U^0_h \|^2_{L^2(\Omega)},  \label{thm_stab_eq01} \\
    & t_M \| U^M_h \|^2_{L^2(\Omega)} + 2 \tau \kappa_1 \sum_{n=1}^M t_n \vertiii{ U^n_h }^2_h + \sum_{n=1}^M t_{n} \| [[ U_h ]]_{n-1} \|^2_{L^2(\Omega)} \le C\| U^0_h \|^2_{L^2(\Omega)},  \label{thm_stab_eq03}  
\end{align}
and if $h^2 \le \gamma \tau$ for some positive constant $\gamma$ small enough, then there exists a constant $C$ such that for any $M\le N$
\begin{align}
    & t_M  \vertiii{ U^M_h }^2_h + \tau \sum_{n=1}^M t_n \| \mathcal{L}^{n}_h U^n_h \|^2_{L^2(\Omega)} +  \tau^{-1} \sum_{n=1}^{M-1} t_n \| [[ U_h ]]_{n} \|^2_{L^2(\Omega)} \le C \| U^0_h \|^2_{L^2(\Omega)},  \label{thm_stab_eq02}
    \end{align}
\end{subequations}
\end{theorem}
\begin{proof}
We prove each of the inequalities above individually.

\vspace{0.1in}

\textit{Proof of \eqref{thm_stab_eq01}}. In \eqref{time_IFE_5}, taking $V_h=U_h$, namely in \eqref{time_IFE_4} setting $V^n_h=U^n_h$, $n=1,2,...,M$, and summing the equalities together, we have
\begin{equation}
\begin{split}
\label{thm_stab_eq1}
& \tau \sum_{n=1}^M a_h(U^n_h,U^n_h) + \sum_{n=1}^M  ( [[U_h]]_{n-1}, U^n_h )_{\Omega} \\
= & \tau \sum_{n=1}^M a_h(U^n_h,U^n_h) + \frac{1}{2} \| U^M_h \|^2_{L^2(\Omega)} +  \frac{1}{2} \sum_{n=1}^M  \| [[ U_h ]]_{n-1} \|^2_{L^2(\Omega)} -  \frac{1}{2} \| U^0_h \|^2_{L^2(\Omega)} = 0
\end{split}
\end{equation}
where we have used the identity $2( [[U_h]]_{n-1}, U^n_h )_{\Omega} = \| U^n_h \|^2_{L^2(\Omega)} +  \| [[ U ]]_{n-1} \|^2_{L^2(\Omega)} - \| U^{n-1}_h \|^2_{L^2(\Omega)}$. Then the coercivity \eqref{thm_coer_eq01} finishes the proof.

\vspace{0.1in}

\textit{Proof of \eqref{thm_stab_eq03}}. The argument is similar. In \eqref{time_IFE_4} setting $V^n_h=U^n_h$, multiplying it by $t_n$ and summing from $n=1,2,...,M$, we have
\begin{equation}
\label{thm_stab_eq1_1}
\tau \sum_{n=1}^M t_n a_h(U^n_h, U^n_h) + \frac{1}{2} \sum_{n=1}^M t_n \| [[ U_h ]]_{n-1} \|^2_{L^2(\Omega)}  + \frac{1}{2} t_M \| U^M_h \|^2_{L^2(\Omega)} = \tau \sum_{n=1}^{M} \| U^{n-1}_h \|^2_{L^2(\Omega)}.
\end{equation}
Then applying the bound of the second term in \eqref{thm_stab_eq01} together with the discrete Poincar\'e inequality in Theorem \ref{thm_dis_poinc} and the coercivity \eqref{thm_coer_eq01} we have the desired result.

\vspace{0.2in}

\textit{Proof of \eqref{thm_stab_eq02}}. First of all we observe the following identities
\begin{subequations}
\label{thm_stab_eq3}
\begin{align}
    & a_h(U^n_h, \mathcal{L}^n_hU^n_h) = ( \mathcal{L}^n_hU^n_h,  \mathcal{L}^n_h U^n_h)_{L^2(\Omega)} = \|  \mathcal{L}^n_h U^n_h \|^2_{L^2(\Omega)}, \label{thm_stab_eq3_1} \\
    &  (U^n_h, \mathcal{L}^n_h U^n_h )_{L^2(\Omega)} = a_h( U^n_h, U^n_h ) = \vertiii{ U^n_h }^2_a, \label{thm_stab_eq3_2} \\
    &  ( \mathcal{R}^n_hU^{n-1}_h, \mathcal{L}^n_h U^n_h )_{L^2(\Omega)} = a_h( \mathcal{R}^n_hU^{n-1}_h, U^n_h ).  \label{thm_stab_eq3_3} 
\end{align}
\end{subequations}
Note that $\mathcal{R}^n_hU^{n-1}_h$ and $U^n_h$ are both in $S^n_h$, we have the following identity
\begin{equation}
\label{thm_stab_eq4}
a_h( \mathcal{R}^n_hU^{n-1}_h, U^n_h ) = \frac{1}{2} \vertiii{ U^n_h }^2_a +  \frac{1}{2} \vertiii{ \mathcal{R}^n_h U^{n-1}_h }^2_a - \frac{1}{2} \vertiii{ U^n_h -  \mathcal{R}^n_h U^{n-1}_h }^2_a.
\end{equation}
Setting $V^n_h = \mathcal{L}^n_h U^n_h$ in \eqref{time_IFE_4}, we then rewrite
\begin{equation}
\label{thm_stab_eq2}
\tau a_h(U^n_h, \mathcal{L}^n_h U^n_h) + (U^n_h, \mathcal{L}^n_h U^n_h)_{L^2(\Omega)} = (U^{n-1}_h, \mathcal{L}^n_h U^n_h)_{L^2(\Omega)} .
\end{equation}
Subtracting the term $( \mathcal{R}^n_hU^{n-1}_h, \mathcal{L}^n_h U^n_h )_{L^2(\Omega)}$ from \eqref{thm_stab_eq2}, and using \eqref{thm_stab_eq3} and \eqref{thm_stab_eq4}, we arrive at the identity
\begin{equation}
\begin{split}
\label{thm_stab_eq5}
& \frac{1}{2} \vertiii{ U^n_h }^2_a - \frac{1}{2} \vertiii{ U^{n-1}_h }^2_a + \tau \|  \mathcal{L}^n_h U^n_h \|^2_{L^2(\Omega)} + \frac{1}{2} \vertiii{ U^n_h -  \mathcal{R}^n_h U^{n-1}_h }^2_a \\
 = & ( (\mathcal{I} - \mathcal{R}^n_h) U^{n-1}_h, \mathcal{L}^n_h U^n_h)_{L^2(\Omega)} +  \frac{1}{2} \vertiii{ \mathcal{R}^n_h U^{n-1}_h }^2_a - \frac{1}{2} \vertiii{ U^{n-1}_h }^2_a .
\end{split}
\end{equation} 
We need to estimate each term in the right hand side of \eqref{thm_stab_eq5}. By orthogonality, boundedness \eqref{thm_coer_eq02}, the estimate in Theorem \eqref{thm_Rn_h_err} we have
\begin{equation}
\begin{split}
\label{thm_stab_eq6}
\vertiii{ \mathcal{R}^n_h U^{n-1}_h }^2_a - \vertiii{ U^{n-1}_h }^2_a &=  -a_h ( \mathcal{R}^n_h U^{n-1}_h - U^{n-1}_h, \mathcal{R}^n_h U^{n-1}_h - U^{n-1}_h )  \le \vertiii{ \mathcal{R}^n_h U^{n-1}_h - U^{n-1}_h }^2_h \\
& \le Ch^2  \| \mathcal{L}^{n-1}_h U^{n-1}_h \|^2_{L^2(\Omega)} + C \tau \vertiii{ U^{n-1}_h }^2_{h}  \\
& \le C \gamma \tau  \| \mathcal{L}^{n-1}_h U^{n-1}_h \|^2_{L^2(\Omega)} + C \tau \vertiii{ U^{n-1}_h }^2_{h} .
\end{split}
\end{equation}
In addition, Theorem \ref{thm_Rn_err} and Young's inequality imply
\begin{equation}
\begin{split}
\label{thm_stab_eq7}
( (\mathcal{I} - \mathcal{R}^n_h) U^{n-1}_h, \mathcal{L}^n_h U^n_h )_{L^2(\Omega)} & \le \| (\mathcal{I} - \mathcal{R}^n_h) U^{n-1}_h \|_{L^2(\Omega)} \| \mathcal{L}^n_h U^n_h \|_{L^2(\Omega)} \\
& \le \left( Ch^2 \| \mathcal{L}^{n-1}_h U^{n-1}_h \|_{L^2(\Omega)}  + C \tau^{1/2} h \vertiii{ U^{n-1}_h }_{h} \right) \| \mathcal{L}^n_h U^n_h \|_{L^2(\Omega)} \\
& \le C h^2 \| \mathcal{L}^{n-1}_h U^{n-1}_h \|^2_{L^2(\Omega)} + C h^2 \| \mathcal{L}^{n}_h U^{n}_h \|^2_{L^2(\Omega)} + C  \tau \vertiii{ U^{n-1}_h }^2_{h} \\
& \le C \gamma \tau \| \mathcal{L}^{n-1}_h U^{n-1}_h \|^2_{L^2(\Omega)} + C \gamma \tau \| \mathcal{L}^{n}_h U^{n}_h \|^2_{L^2(\Omega)} + C  \tau \vertiii{ U^{n-1}_h }^2_{h}.
\end{split}
\end{equation}
Putting \eqref{thm_stab_eq6} and \eqref{thm_stab_eq7} into \eqref{thm_stab_eq5}, we obtain
\begin{equation}
\begin{split}
\label{thm_stab_eq8}
 \frac{1}{2} \vertiii{ U^n_h }^2_a - \frac{1}{2} \vertiii{ U^{n-1}_h }^2_a + \tau \|  \mathcal{L}^n_h U^n_h \|^2_{L^2(\Omega)} \le C \gamma \tau \| \mathcal{L}^{n-1}_h U^{n-1}_h \|^2_{L^2(\Omega)} + C \gamma \tau \| \mathcal{L}^{n}_h U^{n}_h \|^2_{L^2(\Omega)} + C \tau \vertiii{U^{n-1}_h}^2_h.
\end{split}
\end{equation}
Note that $t_n=t_{n-1}+\tau \le 2t_{n-1}$ for $n \ge 2$. Then multiplying \eqref{thm_stab_eq8} by $t_n$ and summing it from $n=2$ to $n=M$ yields
\begin{equation}
\begin{split}
\label{thm_stab_eq9}
 \frac{t_M}{2} \vertiii{ U^M_h }^2_a + (1-C\gamma)  \tau  \sum_{n=2}^M t_n \|  \mathcal{L}^n_h U^n_h \|^2_{L^2(\Omega)} 
  \le C \gamma \tau^2 \| \mathcal{L}^{1}_h U^{1}_h \|^2_{L^2(\Omega)} + C  \tau \sum_{n=2}^M  \vertiii{U^{n-1}_h}^2_h .
 \end{split}
\end{equation}
Special attention needs for $n=1$. Putting \eqref{thm_stab_eq3_1} and \eqref{thm_stab_eq3_2} into \eqref{thm_stab_eq2} with $n=1$, we have
\begin{equation}
\begin{split}
\label{thm_stab_eq10}
\vertiii{ U^1_h }^2_a + \tau \vertii{ \mathcal{L}^1_h U^1_h }^2_{L^2(\Omega)} & = ( U^0_h, \mathcal{L}^1_h U^1_h )_{L^2(\Omega)} \\
& \le \vertii{ U^0_h }_{L^2(\Omega)} \vertii{ \mathcal{L}^1_h U^1_h }_{L^2(\Omega)} \le \tau^{-1} \vertii{ U^0_h }^2_{L^2(\Omega)} + \frac{\tau}{4}  \vertii{ \mathcal{L}^1_h U^1_h }^2_{L^2(\Omega)}.
\end{split}
\end{equation}
Multiplying \eqref{thm_stab_eq10} by $t_1=\tau$ leads to
\begin{equation}
\begin{split}
\label{thm_stab_eq10_1}
\tau \vertiii{ U^1_h }^2_a + \frac{3\tau t_1}{4} \vertii{ \mathcal{L}^1_h U^1_h }^2_{L^2(\Omega)}  \le \vertii{ U^0_h }^2_{L^2(\Omega)}.
\end{split}
\end{equation}
Combining \eqref{thm_stab_eq10_1} with \eqref{thm_stab_eq9}, using the stability \eqref{thm_stab_eq01}(the second term), replacing the norm $\vertiii{\cdot}_a$ by $\vertiii{\cdot}_h$ through the equivalence, and assuming $\gamma$ is small enough such that $1-C\gamma>0$ we obtain the bounds for the first two terms on the left side of \eqref{thm_stab_eq02}. For the third term on the left of \eqref{thm_stab_eq02}, we apply the $L^2$ projection $\mathcal{P}^n_h~:~ W_h \rightarrow S^n_h$ to write
\begin{equation}
\label{thm_stab_eq11}
a_h( U^n_h, \mathcal{P}^n_h[[U_h]]_{n-1} ) = ( \mathcal{L}^n_hU^n_h, \mathcal{P}^n_h[[U_h]]_{n-1} )_{L^2(\Omega)} = ( \mathcal{L}^n_hU^n_h, [[U_h]]_{n-1} )_{L^2(\Omega)}.
\end{equation}
Then in \eqref{time_IFE_4}, taking $V^n_h = \mathcal{P}^n_h [[U_h]]_{n-1}$, and using \eqref{thm_stab_eq11} we have
\begin{equation}
\begin{split}
\label{thm_stab_eq12}
\vertii{ [[ U_h ]]_{n-1} }^2_{L^2(\Omega)} & = ( [[ U_h ]]_{n-1}, [[ U_h ]]_{n-1} - \mathcal{P}^n_h[[ U_h ]]_{n-1} )_{L^2(\Omega)}
 +  ( [[ U_h ]]_{n-1},  \mathcal{P}^n_h[[ U_h ]]_{n-1} )_{L^2(\Omega)} \\
 & = ( [[ U_h ]]_{n-1}, [[ U_h ]]_{n-1} - \mathcal{P}^n_h[[ U_h ]]_{n-1} )_{L^2(\Omega)} - \tau ( \mathcal{L}^n_hU^n_h, [[U_h]]_{n-1} )_{L^2(\Omega)}.
 \end{split}
 \end{equation}
 The first term on the right of \eqref{thm_stab_eq12} can be bounded through the Young's inequality and Theorem \ref{thm_Rn_err}:
 \begin{equation}
 \begin{split}
\label{thm_stab_eq13}
&( [[ U_h ]]_{n-1}, [[ U_h ]]_{n-1} - \mathcal{P}^n_h[[ U_h ]]_{n-1} )_{L^2(\Omega)} = ( [[ U_h ]]_{n-1}, \mathcal{P}^n_h U^{n-1}_h - U^{n-1}_{h}  )_{L^2(\Omega)} \\
 \le & \frac{1}{4} \| [[ U_h ]]_{n-1} \|^2_{L^2(\Omega)} + \| (\mathcal{I} - \mathcal{P}^n_h)U^{n-1}_h \|^2_{L^2(\Omega)} \\
 \le & \frac{1}{4} \| [[ U_h ]]_{n-1} \|^2_{L^2(\Omega)} + Ch^4 \| \mathcal{L}^{n-1}_h U^{n-1}_h \|^2_{L^2(\Omega)} + C\tau h^2 \vertiii{ U^{n-1}_h }^2_h
\end{split}
\end{equation}
where we have also used the smallest distance property $\| (\mathcal{I} - \mathcal{P}^n_h)U^{n-1}_h \|^2_{L^2(
\Omega)} \le \| (\mathcal{I} - \mathcal{R}^n_h)U^{n-1}_h \|^2_{L^2(\Omega)}$. The second term on the right of \eqref{thm_stab_eq12} is also bounded through the Young's inequality:
 \begin{equation}
\label{thm_stab_eq14}
\tau ( \mathcal{L}^n_hU^n_h, [[U_h]]_{n-1} )_{L^2(\Omega)} \le \frac{1}{4} \| [[U_h]]_{n-1} \|^2_{L^2(\Omega)} + \tau^2 \| \mathcal{L}^n_hU^n_h \|^2_{L^2(\Omega)}.
\end{equation}
Substituting \eqref{thm_stab_eq13} and \eqref{thm_stab_eq14} into \eqref{thm_stab_eq12} together with the assumption $h^2\le \gamma \tau$, we obtain
\begin{equation}
\label{thm_stab_eq15}
\tau^{-1} \vertii{ [[ U_h ]]_{n-1} }^2_{L^2(\Omega)} \le C \tau \left( \| \mathcal{L}^{n-1}_h U^{n-1}_h \|^2_{L^2(\Omega)} +  \| \mathcal{L}^{n}_h U^{n}_h \|^2_{L^2(\Omega)} + \vertiii{ U^{n-1}_h }^2_h \right).
\end{equation}
Now we multiply \eqref{thm_stab_eq15} by $t_{n-1}$ and note $t_{n-1}\le t_n$. Then summing the resulted inequalities from $n=2$ to $M$ and using the bounds for second terms in \eqref{thm_stab_eq03} and \eqref{thm_stab_eq02}, we arrive at the estimate for the third term in \eqref{thm_stab_eq02}.
\end{proof}

\begin{remark}
\label{rem_stability_1}
We note that one of the keys in the proof of \eqref{thm_stab_eq02} is the employment of the norm $\vertiii{ \cdot }_a$ induced from $a_h(\cdot,\cdot)$ in the identity \eqref{thm_stab_eq5}. Roughly speaking if it is replaced by the energy norm $\vertiii{\cdot}_h$, then the coercivity and boundedness in Theorem \ref{thm_coer} have to be used to bound $a_h(U^n_h,U^n_h)$ and $a_h(U^{n-1}_h, U^{n-1}_h)$ which can not give the same coefficients for $\vertiii{ U^n_h }_h$ and $\vertiii{ U^{n-1}_h }_h$ as \eqref{thm_stab_eq5}, and thus one can not do cancellation when summing these identities as \eqref{thm_stab_eq9}. Moreover, the estimate \eqref{thm_stab_eq6} relies on the orthogonality property of $a_h(\cdot,\cdot)$, and one order will be lost if $\vertiii{\cdot}_a$ is replaced by $\vertiii{\cdot}_h$ in that estimate. So we think the uniform degrees of freedom and weak form in dynamics does not only benefit computation but also analysis.
\end{remark}

\begin{remark}
\label{rem_stability_2}
Similar stability results are also derived in \cite{2013Zunino} to analyze XFEM for moving interface problems. However their approach relies on certain assumptions on the interpolation errors between the extended finite element spaces and the standard continuous finite element spaces, see (11)-(15) in \cite{2013Zunino}. As mentioned in the article, the rigorous proof of those assumptions is still an open problem which we think are also difficult to prove even for the IFE spaces.
\end{remark}

Next we proceed to estimate the fully discrete errors. Given the IFE solution $U_h$ to the scheme \eqref{time_IFE_5} or \eqref{time_IFE_4}, we define the total error 
\begin{equation}
\label{error}
E_h := u - U_h \in \mathbb{W}_h.
\end{equation}
We follow the argument of \cite{2013Zunino} to show the following estimate on the consistency error.
\begin{theorem}
\label{thm_consist}
Suppose the exact solution has the regularity $u\in L^2(0,T;\widetilde{H}^2(\Omega))\cap H^1(0,T;H^1(\Omega^-\cup \Omega^+))\cap H^2(0,T;L^2(\Omega))$, let $U_h$ be the IFE solution to \eqref{time_IFE_5} or \eqref{time_IFE_4}, and let $\mathcal{V}\in L^{\infty}(\Gamma(t))$ with $\|\mathcal{V}(t)\cdot\bfn\|_{L^{\infty}(\Gamma(t))}\le K$, $\forall t$, then for any $V_h\in \mathbb{S}_h$ and $\epsilon>0$ there holds
\begin{equation}
\begin{split}
\label{thm_consist_eq0}
A_h( E_h, V_h ) & \le  C \tau^2\epsilon^{-1} \left( \| \partial_{tt}u \|^2_{L^2(0,T;L^2(\Omega))} + K \| \partial_t u \|^2_{L^2(0,T;H^1(\Omega))}  \right) \\
&+ \frac{ \epsilon }{2} \left(  \max_{n=1,...,N} \| V^n_h \|^2_{L^2(\Omega)} + K\tau \sum_{n=1}^N \| V^n_h \|^2_{H^1(\Omega)} \right) + C\epsilon^{-1}h^4 \| u_0 \|^2_{H^2(\Omega)}.
\end{split}
\end{equation}
\end{theorem}
\begin{proof}
By the assumption $f=0$, noticing the regularity of $u$, i.e., $[u]_e=0$, $\forall e\in\mathring{\mathcal{E}}_h$, $[[u]]_{n-1}=0$ for $n=2,...,N$, and using \eqref{time_IFE_5}-\eqref{time_IFE_7} we have
\begin{equation}
\begin{split}
\label{thm_consist_eq1}
& A_h(E_h,V_h)  = A_h( u, V_h ) - A_h( U_h, V_h )  = A_h( u, V_h ) - ( U^0_h, V^1_h )_{L^2(\Omega)}  \\
 = & \sum_{n=1}^N \left( \int_{J_n} ( \partial_t u, V_h )_{L^2(\Omega)}  +  (\sqrt{\beta} \nabla u(t^-_{n}) ,\sqrt{\beta} \nabla V^n_h )_{L^2(\Omega)} -  \sum_{e\in\mathring{\mathcal{E}}_h} ( {\beta \nabla u(t^-_{n}) \cdot \bfn}, [V^n_h]_e )_{L^2(e)} \right) dt \\
 &  + (u(t_0)-U^0_h, V^1_h)_{L^2(\Omega)}
\end{split}
\end{equation}
where we denote the first summation above by $I$ and the second term $(u(t_0)-U^0_h,V^1_h)_{L^2(\Omega)}$ by $II$. Using the equation for $u$ and applying integration by parts we have
\begin{equation}
\label{thm_consist_eq2}
I = \sum_{n=1}^N \int_{J_n} ( \beta \triangle u(t), V_h )_{L^2(\Omega)} dt - \tau ( \beta \triangle u(t^-_{n}), V_h )_{L^2(\Omega)}.
\end{equation}
We introduce a function $\mathcal{G}(t)=(\beta \triangle u(t), V_h)_{L^2(\Omega)}$. By the mean value theorem there exists $z_{n}\in J_n=[t_{n-1},t_n]$ such that $I$ in \eqref{thm_consist_eq2} can be expressed into
\begin{equation}
\begin{split}
\label{thm_consist_eq3}
I = \sum_{n=1}^N \int_{J_n} \mathcal{G}(t) dt - \tau \mathcal{G}(t^-_{n}) = \sum_{n=1}^N \tau \mathcal{G}(z_{n}) - \tau \mathcal{G}(t^-_{n}) = \tau \sum_{n=1}^N \int_{t_{n}}^{z_n} \frac{d}{dt} \mathcal{G}(t) dt \le \tau \sum_{n=1}^N \int_{J_n} \verti{ \frac{d}{dt} \mathcal{G}(t) } dt.
 \end{split}
\end{equation}
Now we need to estimate $\verti{ \frac{d}{dt} \mathcal{G}(t) }$. For this purpose, we split the integral on $\Omega$ into the integrals on $\Omega^{\pm}(t)$ which are evolving with respect to time. The temporal derivative of the domain integral is based on the formula \eqref{functional_2}:
\begin{equation}
\begin{split}
\label{thm_consist_eq4}
\verti{ \frac{d}{dt} \mathcal{G}(t) } & =  \verti{  \sum_{s=\pm} \frac{d}{dt} \int_{\Omega^s(t)} \beta \triangle u V_h dX } \\
& =  \verti{  \sum_{s=\pm}  \int_{\Omega^s(t)} \partial_t \left( \beta \triangle u V_h \right) dX +  \int_{\Gamma(t)}  \beta \triangle u V_h \mathcal{V}\cdot\bfn ds }  \\
& =  \verti{  \sum_{s=\pm}  \int_{\Omega^s(t)} \partial_{tt} u V_h dX +  \int_{\Gamma(t)}   \partial_t u V_h \mathcal{V}\cdot\bfn ds } \\
& \le   \int_{\Omega} \verti{  \partial_{tt} u V_h } dX + K \int_{\Gamma(t)} \verti{   \partial_t u V_h } ds.
\end{split}
\end{equation}
Putting \eqref{thm_consist_eq4} into \eqref{thm_consist_eq3} and applying the Young's inequality, we first have
\begin{equation}
\begin{split}
\label{thm_consist_eq5}
\tau \sum_{n=1}^N \int_{J_n}\int_{\Omega} \verti{  \partial_{tt} u V_h } dX dt 
& \le \tau  \sum_{n=1}^N \int_{J_n}  \vertii{  \partial_{tt} u }_{L^2(\Omega)}  \|V_h\|_{L^2(\Omega)}   dt  \\
& = \tau  \sum_{n=1}^N  \|V^n_h\|_{L^2(\Omega)} \int_{J_n}  \vertii{  \partial_{tt} u }_{L^2(\Omega)}    dt  \\
& \le  T^{1/2}\max_{n=1,...,N} \|V^n_h\|_{L^2(\Omega)} \tau  \vertii{  \partial_{tt} u }_{L^2(0,T;L^2(\Omega))}   \\
& \le  T \tau^2 \epsilon^{-1}  \vertii{  \partial_{tt} u }^2_{L^2(0,T;L^2(\Omega))} + \frac{\epsilon}{4}  \max_{n=1,...,N} \|V^n_h\|^2_{L^2(\Omega)}
\end{split}
\end{equation}
where $T$ is the total time. Using the trace inequality from $\Gamma(t)$ to $\Omega$ given by Theorem \ref{thm_glob_trace} and the standard trace inequality, we can use a similar argument above to get the bound
\begin{equation}
\begin{split}
\label{thm_consist_eq6}
K\tau  \sum_{n=1}^N \int_{J_n} \int_{\Gamma(t)} \verti{   \partial_t u V_h } ds dt 
& \le K  \sum_{n=1}^N \int_{J_n} \tau \vertii{ \partial_t u }_{L^2(\Gamma(t))} \vertii{ V_h }_{L^2(\Gamma(t))}  dt \\
& \le K  \sum_{n=1}^N \int_{J_n} C \tau \vertii{ \partial_t u }_{H^1(\Omega)} \vertii{ V_h }_{H^1(\Omega)}  dt \\
& \le K  \sum_{n=1}^N \int_{J_n} C \tau^2 \epsilon^{-1} \vertii{ \partial_t u }^2_{H^1(\Omega)} + \frac{ \epsilon }{4} \vertii{ V_h }^2_{H^1(\Omega)}  dt \\
& = CK \tau^2 \epsilon^{-1} \| \partial_t u \|^2_{L^2(0,T;H^1(\Omega))} +  \frac{ K \epsilon \tau}{4} \sum_{n=1}^N \| V^n_h \|^2_{H^1(\Omega)}.
\end{split}
\end{equation}
\eqref{thm_consist_eq5} and \eqref{thm_consist_eq6} give the bound of $I$. In addition, the term $II$ can be directly bounded by the Young's inequality and Theorem \ref{thm_ellip_proj_err}:
\begin{equation}
\begin{split}
\label{thm_consist_eq7}
II & = (u_0 - U^0_h, V^1_h)_{L^2(\Omega)} = ( u_0 - \mathcal{R}^0_h u_0, V^1_h)_{L^2(\Omega)} \le \| u_0 - \mathcal{R}^0_h u_0 \|_{L^2(\Omega)} \| V^1_h \|_{L^2(\Omega)} \\
& \le  \epsilon^{-1} \| u_0 - \mathcal{R}^0_h u_0 \|^2_{L^2(\Omega)} + \frac{\epsilon}{4} \| V^1_h \|^2_{L^2(\Omega)}
\le C \epsilon^{-1} h^4 \| u_0  \|^2_{H^2(\Omega)} + \frac{\epsilon}{4} \| V^1_h \|^2_{L^2(\Omega)}.
\end{split}
\end{equation}
Combing it with the estimate of the term $II$ we have the desired result.
\end{proof}

An alternative expression of the bilinear form $A_h(\cdot,\cdot)$ to \eqref{time_IFE_6} is also needed.

\begin{lemma}
\label{lemma_Ah}
For every $U_h \in \mathbb{W}_h$ and $V_h\in \mathbb{S}_h$, there holds
\begin{equation}
\begin{split}
\label{lemma_Ah_eq0}
A_h(U_h,V_h) &=  \tau \sum_{n=1}^N a_h(U_h(t^-_n),V_h(t^-_n))  - \sum_{n=1}^{N-1} ( U_h(t^-_n), [[V_h]]_n )_{L^2(\Omega)} + (U_h(t^-_N), V_h(t^-_N))_{L^2(\Omega)}.
\end{split}
\end{equation}
\end{lemma}
\begin{proof}
Using the integration by parts for the temporal direction, we have
\begin{equation}
\begin{split}
\label{lemma_Ah_eq1}
& \int_{J_n} \int_{\Omega} \partial_t U_h V_h dX dt  = \int_{\Omega} \int^{t_n}_{t_{n-1}} \partial_t U_h V_h dt dX \\
 = & (U_h(t^-_n),V_h(t^-_n))_{L^2(\Omega)} - (U_h(t^+_{n-1}),V_h(t^+_{n-1}))_{L^2(\Omega)} - \int_{J_n} ( U_h, \partial_t V_h )_{L^2(\Omega)} dt
\end{split}
\end{equation}
where the last term vanishes since $V_h\in \mathbb{S}_h$. Then we note the following identity
\begin{equation}
\begin{split}
\label{lemma_Ah_eq2}
& \sum_{n=1}^N  (U_h(t^-_n),V_h(t^-_n))_{L^2(\Omega)} - (U_h(t^+_{n-1}),V_h(t^+_{n-1}))_{L^2(\Omega)} \\
= &  (U_h(t^-_N),V_h(t^-_N))_{L^2(\Omega)} -  (U_h(t^+_0),V_h(t^+_0))_{L^2(\Omega)} \\
- &  \sum_{n=1}^{N-1} \left(  (U_h(t^+_n), V_h(t^+_n))_{L^2(\Omega)} -  (U_h(t^-_n), V_h(t^-_n))_{L^2(\Omega)} \right) \\
= &  (U_h(t^-_N),V_h(t^-_N))_{L^2(\Omega)} -  (U_h(t^+_0),V_h(t^+_0))_{L^2(\Omega)} \\
- &  \sum_{n=1}^{N-1} \left(  ( [[U_h]]_n, V_h(t^+_n))_{L^2(\Omega)} +  (U_h(t^-_n), [[V_h]]_n)_{L^2(\Omega)} \right). 
\end{split}
\end{equation}
In \eqref{lemma_Ah_eq2} we further note $V_h(t^+_{n})=V_h(t^-_{n+1})$ since $V_h\in \mathbb{S}_h$. Putting \eqref{lemma_Ah_eq2} into \eqref{time_IFE_6} yields the desired result.
\end{proof}

Now based on the estimates prepared above, we can use the duality argument to analyze the solution errors. This idea was introduced in \cite{1991ErikssonJohnsonI,1995ErikssonClaes} for time-dependent adaptive mesh methods. 

\begin{theorem}
\label{thm_label_err_L2}
Under the conditions of Theorem \ref{thm_stab}, suppose the exact solution satisfies the regularity $u\in L^2(0,T;\widetilde{H}^2(\Omega)) \cap H^1(0,T;H^1(\Omega^-\cup \Omega^+))\cap H^2(0,T;L^2(\Omega)) \cap L^{\infty}(0,T; \widetilde{H}^2(\Omega))$, let $U_h$ be the IFE solution to \eqref{time_IFE_5} or \eqref{time_IFE_4} and let $\mathcal{V}\in L^{\infty}(\Gamma(t))$ with $\|\mathcal{V}(t)\cdot\bfn\|_{L^{\infty}(\Gamma(t))}\le K$, $\forall t$, then there holds
\begin{equation}
\begin{split}
\label{thm_label_err_L2_eq0}
\| u(t_N) - U^N_h \|_{L^2(\Omega)} + h| u(t_N) - U^N_h |_{H^1(\Omega)} & \le (1+ \sqrt{K}) C \tau \left( \| \partial_{tt}u \|_{L^2(0,T;L^2(\Omega))} +  \| \partial_t u \|_{L^2(0,T;H^1(\Omega))}  \right) \\
& + \sqrt{\log(1+N)} C h^2\| u \|_{L^{\infty}(0,T;H^2(\Omega))} .
\end{split}
\end{equation}
\end{theorem}
\begin{proof}
First of all, we show the estimate in the $L^2$-norm. By the discrete duality argument to \eqref{time_IFE_4}, since $a_h(\cdot,\cdot)$ is coercive on every $S^n_h$, given each $Z^{N+1}_h\in W_h$ we can define a sequence $Z^n_h\in S^n_h$, $n=N,...,1$ such that
\begin{equation}
\label{thm_label_err_L2_eq1}
\tau a_h(V^n_h,Z^n_h) + ( V^n_h, Z^{n}_h )_{L^2(\Omega)} =  ( V^n_h, Z^{n+1}_h )_{L^2(\Omega)}, ~~~ \forall V^n_h\in S^n_h.
\end{equation}
Using the expression in Lemma \ref{lemma_Ah} we can write the equivalent format to \eqref{thm_label_err_L2_eq1} in terms of the bilinear form $A_h(\cdot,\cdot)$ by summing \eqref{thm_label_err_L2_eq1} from $n=N,...,1$, namely we need to find $Z_h \in \mathbb{S}_h$ such that
\begin{equation}
\label{thm_label_err_L2_eq2}
A_h(V_h, Z_h ) = (V^N_h, Z^{N+1}_h)_{L^2(\Omega)}, ~~~~ \forall V_h \in  \mathbb{S}_h.
\end{equation}
Let's employ the error decomposition similar to \eqref{err_decomp}, and define the corresponding functions $\xi_h\in \mathbb{S}_h$ with $\xi_h|_{J_n}=\xi^n_h$ and $\eta_h\in \mathbb{W}_h$ with $\eta_h|_{J_n}=u(t) - \mathcal{R}^n_hu$, $n=1,...,N$ which leads to $E_h=\xi_h + \eta_h$ with $E_h$ defined in \eqref{error}. In particular we also note that $\eta_h(t^-_{n})=u(t_{n}) - \mathcal{R}^n_hu = \eta^n_h$. Letting $Z^{N+1}_h = \xi^N_h$ and $V_h = \xi_h$ in \eqref{thm_label_err_L2_eq2} and using Lemma \ref{lemma_Ah} we have the function $Z_h$ satisfying
\begin{equation}
\begin{split}
\label{thm_label_err_L2_eq4}
 &\| \xi^N_h \|^2_{L^2(\Omega)}  = A_h( \xi_h, Z_h ) = A_h(E_h,Z_h) - A_h(\eta_h,Z_h) \\
 = & A_h(E_h,Z_h) - \tau \sum_{n=1}^N a_h( \eta^n_h,Z^n_h)   +  \sum_{n=1}^{N-1} ( \eta^n_h , [[Z_h]]_{n} )_{L^2(\Omega)} -  (\eta^N_h, Z^N_h)_{L^2(\Omega)},
\end{split}
\end{equation}
where the terms in the right hand side are denoted by $Q_i$, $i=1,2,3,4$, respectively. Now we proceed to estimate each term $Q_i$ individually. First of all, we have $\| V^n_h \|_{H^1(\Omega)}\le C \vertiii{V^n_h}_h$ by the discrete Poincar\'e inequality in Theorem \ref{thm_dis_poinc}. Then applying the counterpart of \eqref{thm_stab_eq01} in Theorem \ref{thm_stab} for the sequence $Z^n_h$ together with Theorem \ref{thm_consist}, we obtain
\begin{equation}
\begin{split}
\label{thm_label_err_L2_eq5}
Q_1 = A_h( E_h, V_h ) \le & C \tau^2\epsilon^{-1} \left( \| \partial_{tt}u \|^2_{L^2(0,T;L^2(\Omega))} + K \| \partial_t u \|^2_{L^2(0,T;H^1(\Omega))}  \right) \\
&+ \frac{\epsilon C(K+1)}{2} \| Z^{N+1}_h \|^2_{L^2(\Omega)} + C\epsilon^{-1}h^4 \| u_0 \|^2_{H^2(\Omega)},
\end{split}
\end{equation}
where $Z^{N+1}_h = \xi^N_h$. Next we note that $Z^n_h\in S^n_h$ and thus 
\begin{equation}
\label{thm_label_err_L2_eq6}
Q_2 = \tau\sum_{n=1}^N a_h(\eta^n_h,Z^n_h) = \tau\sum_{n=1}^N a_h( u(t_n) - \mathcal{R}^n_hu, Z^n_h ) = 0.
\end{equation}
By Schwarz inequality, using the last term in \eqref{thm_stab_eq02} of Theorem \ref{thm_stab} (also the counterpart for the sequence $Z^n_h$) and applying the estimate for $\eta^n_h$, we can bound $Q_3$ by
\begin{equation}
\begin{split}
\label{thm_label_err_L2_eq7}
Q_3 & \le \sum_{n=1}^{N-1} \| \eta^n_h \|_{L^2(\Omega)} \| [[Z_h]]_n \|_{L^2(\Omega)} \\
& \le \left( \sum_{n=1}^{N-1} \tau t^{-1}_n \| \eta^n_h \|^2_{L^2(\Omega)} \right)^{1/2} \left( \sum_{n=1}^{N-1} \tau^{-1} t_n \| [[Z_h]]_n \|^2_{L^2(\Omega)}\right)^{1/2} \\
& \le \max_{n=1,...,N}  \| \eta^n_h \|_{L^2(\Omega)}\left( \sum_{n=1}^{N-1} \tau \frac{1}{n \tau} \right)^{1/2} C \| Z^{N+1}_h \|_{L^2(\Omega)} \\
& \le C \sqrt{\log(1+N)} h^2 \| u \|_{L^{\infty}(0,T; H^2(\Omega))} \| \xi^N_h \|_{L^2(\Omega)} \\
& \le C \log(1+N) h^4 \epsilon^{-1} \| u \|^2_{L^{\infty}(0,T; H^2(\Omega))} + \frac{\epsilon}{4} \| \xi^N_h \|^2_{L^2(\Omega)}.
\end{split}
\end{equation}
The last term $Q_4$ can be bounded by the estimate for $\eta^N_h$ and the first term in \eqref{thm_stab_eq01} of Theorem \ref{thm_stab}:
\begin{equation}
\begin{split}
\label{thm_label_err_L2_eq8}
Q_4 & \le \| \eta^N_h \|_{L^2(\Omega)} \| Z^N_h \|_{L^2(\Omega)} \le Ch^2 \| u \|_{L^{\infty}(0,T;H^2(\Omega))} \| \xi^{N}_h \|_{L^2(\Omega)} \le Ch^4 \epsilon^{-1} \| u \|^2_{L^{\infty}(0,T;H^2(\Omega))} + \frac{\epsilon}{4} \| \xi^{N}_h \|^2_{L^2(\Omega)} . 
\end{split}
\end{equation}
Substituting \eqref{thm_label_err_L2_eq5}-\eqref{thm_label_err_L2_eq8} into \eqref{thm_label_err_L2_eq4}, we finally obtain
\begin{equation}
\begin{split}
\label{thm_label_err_L2_eq9}
\left( 1 - C(K+1)\epsilon \right) \| \xi^N_h \|^2_{L^2(\Omega)} & \le C \tau^2\epsilon^{-1} \left( \| \partial_{tt}u \|^2_{L^2(0,T;L^2(\Omega))} + K \| \partial_t u \|^2_{L^2(0,T;H^1(\Omega))}  \right) \\
&+ C \log(1+N) h^4 \epsilon^{-1} \| u \|^2_{L^{\infty}(0,T; H^2(\Omega))}.
\end{split}
\end{equation}
Choosing $\epsilon$ sufficiently small such that $1-C(K+1)\epsilon>0$, we have the error bound for $\xi^N_h$. Combining it with the estimate for $\eta^N_h$, we have the desired estimate for the $L^2$-norm. Finally the $H^1$-norm estimate simply follows from the inverse estimate $|\xi^N_h|_{H^1(T)}\le Ch_T^{-1} \|\xi^N_h\|_{L^2(T)}$ by \eqref{inver_inequa} for each element $T$ since $\xi^N_h\in S^N_h$.
\end{proof}

\begin{remark}
\label{rem_regularity}
We comment on the results between Theorem \ref{thm_label_err_L2} and the standard finite element method solving the stationary parabolic interface problem in \cite{1998ChenZou}. We first note that the regularity assumptions in Theorem \ref{thm_label_err_L2} on the temporal direction are stronger than Theorems 3.2 and 3.3 in \cite{1998ChenZou}. The regularity of parabolic interface problems with a stationary interface is discussed in \cite{2002HuangZou} which is also indeed weaker than those of Theorem \ref{thm_label_err_L2}. But we do not know of any literature where the regularity of parabolic equations with moving interface is studied. In addition it is interesting to note that there is also a ``log" term in Theorems 3.2 and 3.3 in \cite{1998ChenZou} but on the spatial mesh size $h$, i.e., $\log(h)$, while the ``log" term in Theorem \ref{thm_label_err_L2} appears to be on the temporal step size $\tau$, i.e., $|\log{(1+N)}|\approx |\log{(\tau)}|$. In general $\tau$ is taken in some order of $h$ to guarantee convergence, then these two results are actually comparable. However we should also read from these results that for the IFE method based on unfitted meshes the errors in the spatial direction and temporal direction are not completely decoupled due to the bound $\sqrt{\log(1+N)}h^2$, but as usual $\sqrt{\log(1+N)}$ has very limited affect on the total error.
\end{remark}

\begin{remark}
\label{rem_const}
Note that the generic constant in Theorem \ref{thm_label_err_L2} is time-dependent, i.e., $C=C(T)$. The source of the time dependence comes from Theorems \ref{thm_ellip_proj_err}, \ref{thm_Rn_err} and \ref{thm_dis_poinc} based on the duality argument that involves the constants of the elliptic regularity \cite{2010ChuGrahamHou,2002HuangZou,1987Leguillon}. The result in \cite{2010ChuGrahamHou} states that the regularity constant depends on the distance from the interface to the domain boundary. Moreover, the analysis in \cite{1987Leguillon} shows singularity may occur if the interface touches the boundary. However, to our best knowledge, there is no work in the literation giving detailed analysis on how these regularity constants depend on the interface geometry. Since the geometry can be really arbitrary during the motion of interface which is different from the stationary interface problems, a rigorous geometric analysis on the regularity constants can be important and interesting.
\end{remark}


\section{Numerical Experiments} In this section, we present a group of numerical experiments to validate the theoretical analysis above. Note that some exploratory numerical experiments were given in \cite{2013HeLinLinZhang,2013LinLinZhang1}, but the IFE method they used does not include the penalties on interface edges which is then shown to only produce suboptimal convergent solutions for elliptic interface problems \cite{2015LinLinZhang} and thus can not be expected to be a good choice for moving interface problems. Here we consider a domain $\Omega=(-1,1)\times(-1,1)$ with three types of moving interface:
\begin{subequations}
\label{interf}
\begin{align}
     \text{(a translating line)}& ~ \Gamma_1(t): \varphi_1 =0 ~~ \text{with}~ \varphi_1 = x- (\pi/5+t ),   \\
     \text{(a moving circle)} &~ \Gamma_2(t):  \varphi_2 =0 ~~ \text{with}~ \varphi_2= (x-0.3\cos(\pi t))^2 +  (y-0.3\sin(\pi t))^2 - (\pi/6)^2, \\
     \text{(a rotating ellipse)} &~  \Gamma_3(t): \varphi_3 =0 ~~ \text{with}~\varphi_3 = 16(\cos(\pi t)x + \sin(\pi t)y)^2 + 49(-\sin(\pi t)x + \cos(\pi t)y)^2 - \pi^2,
\end{align}
\end{subequations}
which are illustrated in Figure \ref{fig:move_interf} where the red solid line is the interface curve and the blue dashed line denotes the trajectory of the center/focus. Here we mention that rotation motion widely appears in fluid-structure-interaction (FSI), for instance the vibration of turbine blades impacted by the fluid flow \cite{2020LanRamirezSun}. The considered situations can be all considered as large rotational/translational motions which in general can cause elements to become ill-shaped and thus reduce the accuracy of numerical solutions for some conventional moving mesh methods. For each of these interfaces and their motions, we define the subdomains $\Omega^+_i(t)=\{X\in \Omega:\varphi_i(X,t)>0\}$ and $\Omega^-_i(t)=\{X\in\Omega:\varphi_i(X,t)<0\}$, $i=1,2,3$, fix $\beta^-=1$ and $\beta^+=10$, and define the corresponding analytical solutions as
\begin{subequations}
\label{ana_solu}
\begin{align}
      & u_1(X,t) = \sin(x- (\pi/5+t ))/\beta^{\pm}~~ \text{in} ~ \Omega^{\pm}_1(t),   \\
      & 
     u_2(X,t) = 
     \begin{cases}
      & \frac{((x-0.3\cos(\pi t))^2 +  (y-0.3\sin(\pi t))^2)^{5/2}(\pi/6)^{-1}}{\beta^{-}} ~~ \text{in} ~ \Omega^{-}_2(t), \\ 
      &  \frac{((x-0.3\cos(\pi t))^2 +  (y-0.3\sin(\pi t))^2)^{5/2}(\pi/6)^{-1}}{\beta^{+}} +(\pi/6)^{4}(\frac{1}{\beta^-} - \frac{1}{\beta^+} ) ~~ \text{in} ~ \Omega^{+}_2(t),
\end{cases} \\
      & u_3(X,t) = 
     \begin{cases}
      & \frac{(\pi/4)^2(\pi/7)^2}{\beta^{-}}\left( \frac{(\cos(\pi t)x + \sin(\pi t)y)^2}{(\pi/4)^2} + \frac{(-\sin(\pi t)x + \cos(\pi t)y)^2}{(\pi/7)^2} \right)^{5/2}~~ \text{in} ~ \Omega^{-}_3(t), \\
      &(\frac{\pi}{4})^2(\frac{\pi}{7})^2 \left( \frac{1}{\beta^{+}}\left( \frac{(\cos(\pi t)x + \sin(\pi t)y)^2}{(\pi/4)^2} + \frac{(-\sin(\pi t)x + \cos(\pi t)y)^2}{(\pi/7)^2} \right)^{5/2} + (\frac{1}{\beta^{-}} - \frac{1}{\beta^{+}}) \right)  ~~ \text{in} ~ \Omega^{+}_3(t).
\end{cases}
\end{align}
\end{subequations}

\begin{figure}[h]
\centering
\begin{subfigure}{.3\textwidth}
     \includegraphics[width=2.1in]{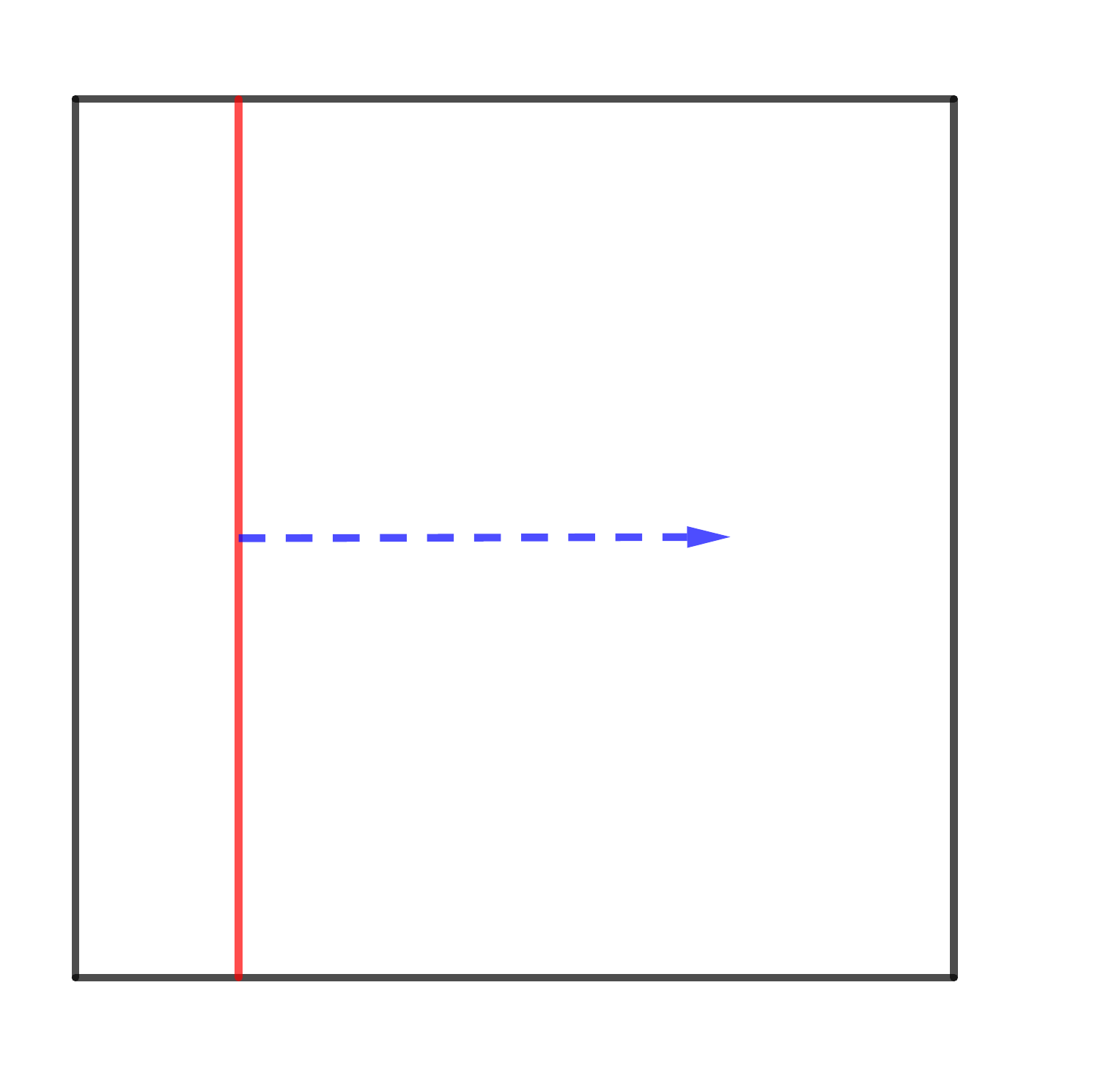}
     \label{trans_line} 
\end{subfigure}
~
\begin{subfigure}{.3\textwidth}
     \includegraphics[width=2.1in]{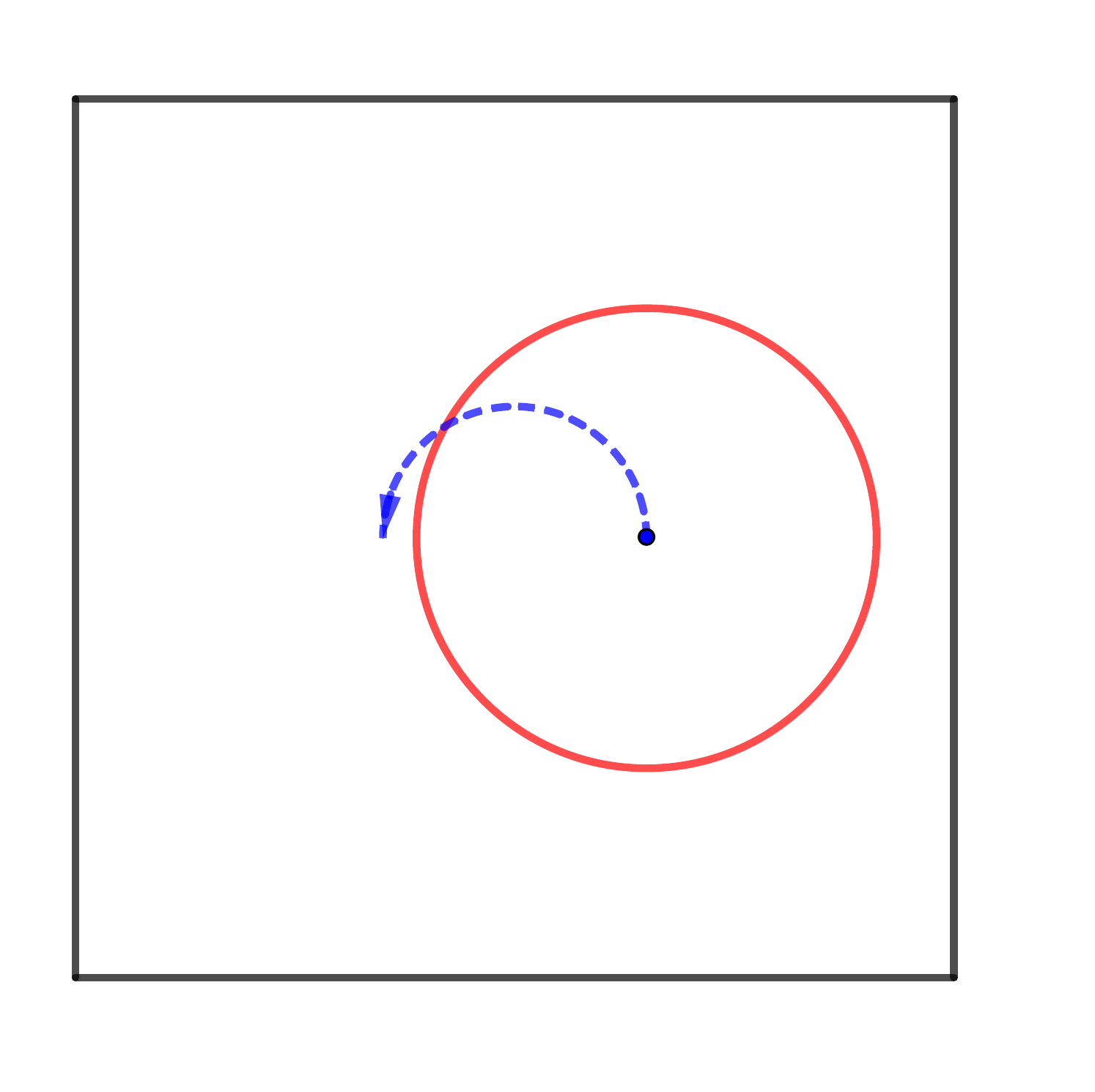}
     \label{move_cir} 
\end{subfigure}
~
 \begin{subfigure}{.3\textwidth}
     \includegraphics[width=2.1in]{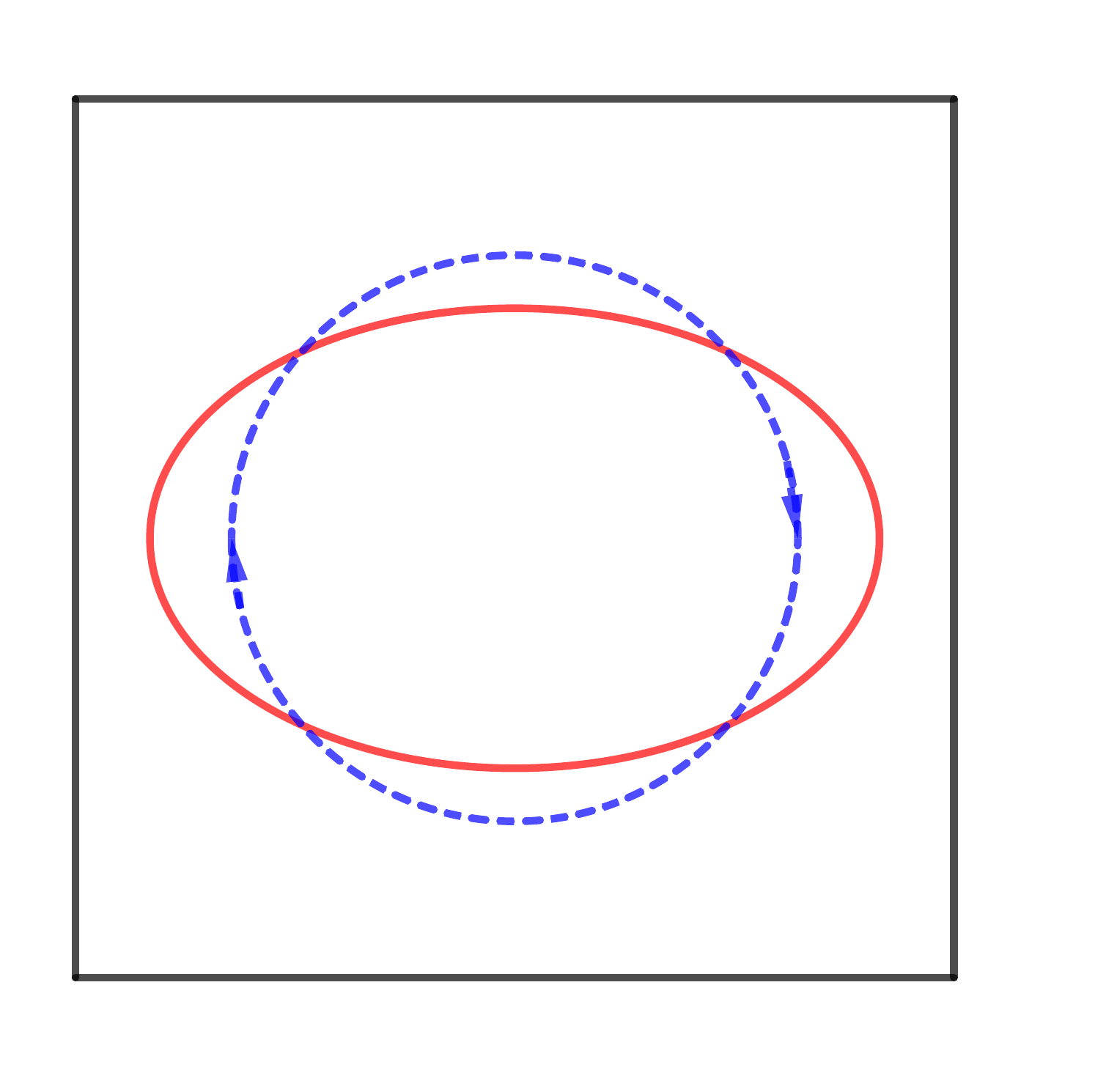}
     \label{rot_elli} 
\end{subfigure}
     \caption{The interface and movement: a translating line (left), a moving circle (middle) and a rotating ellipse (right).}
  \label{fig:move_interf} 
\end{figure}

\begin{figure}[h]
\centering
\begin{subfigure}{.32\textwidth}
     \includegraphics[width=2.2in]{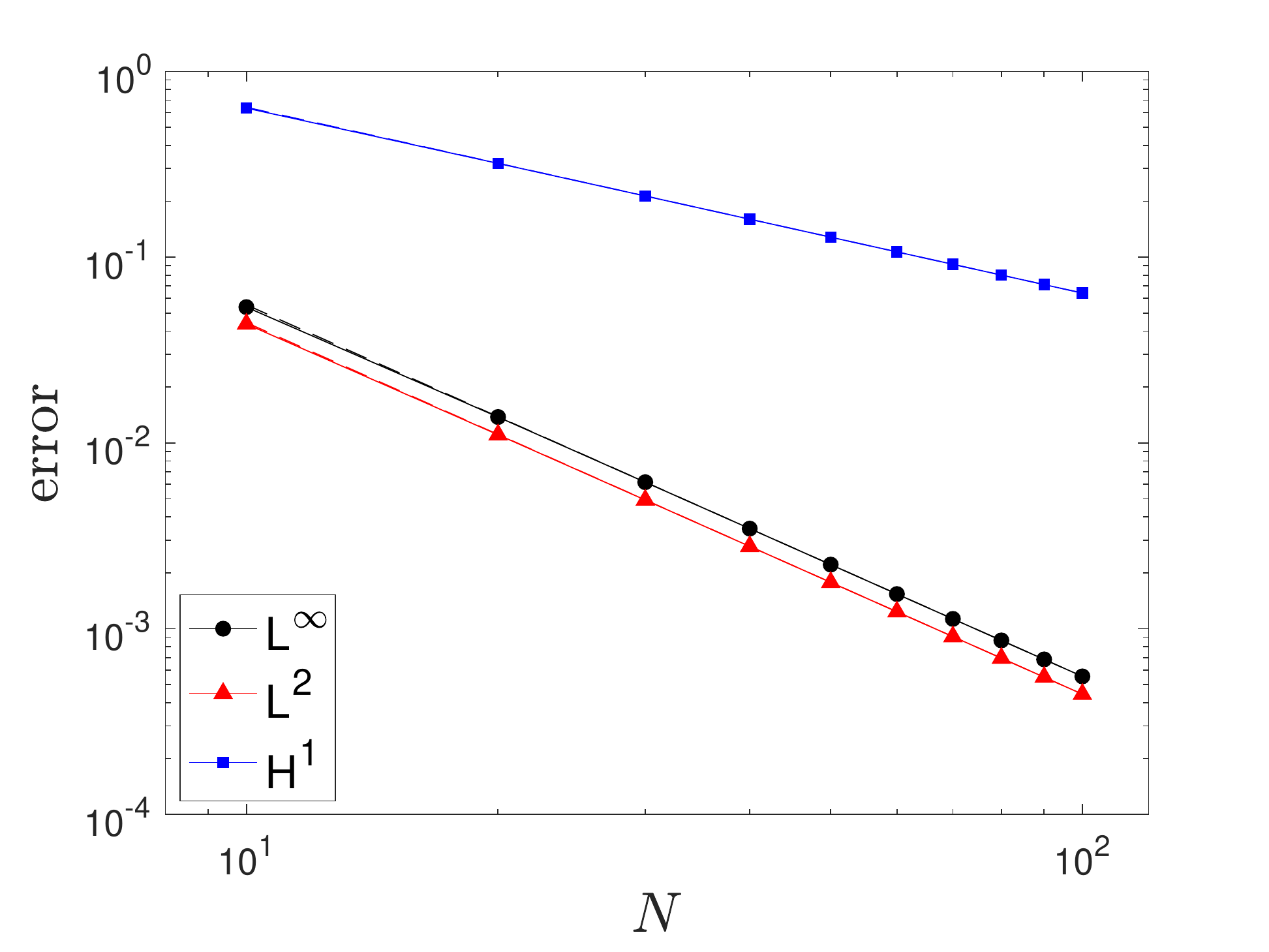}
     \label{trans_line_error} 
\end{subfigure}
~
\begin{subfigure}{.32\textwidth}
     \includegraphics[width=2.2in]{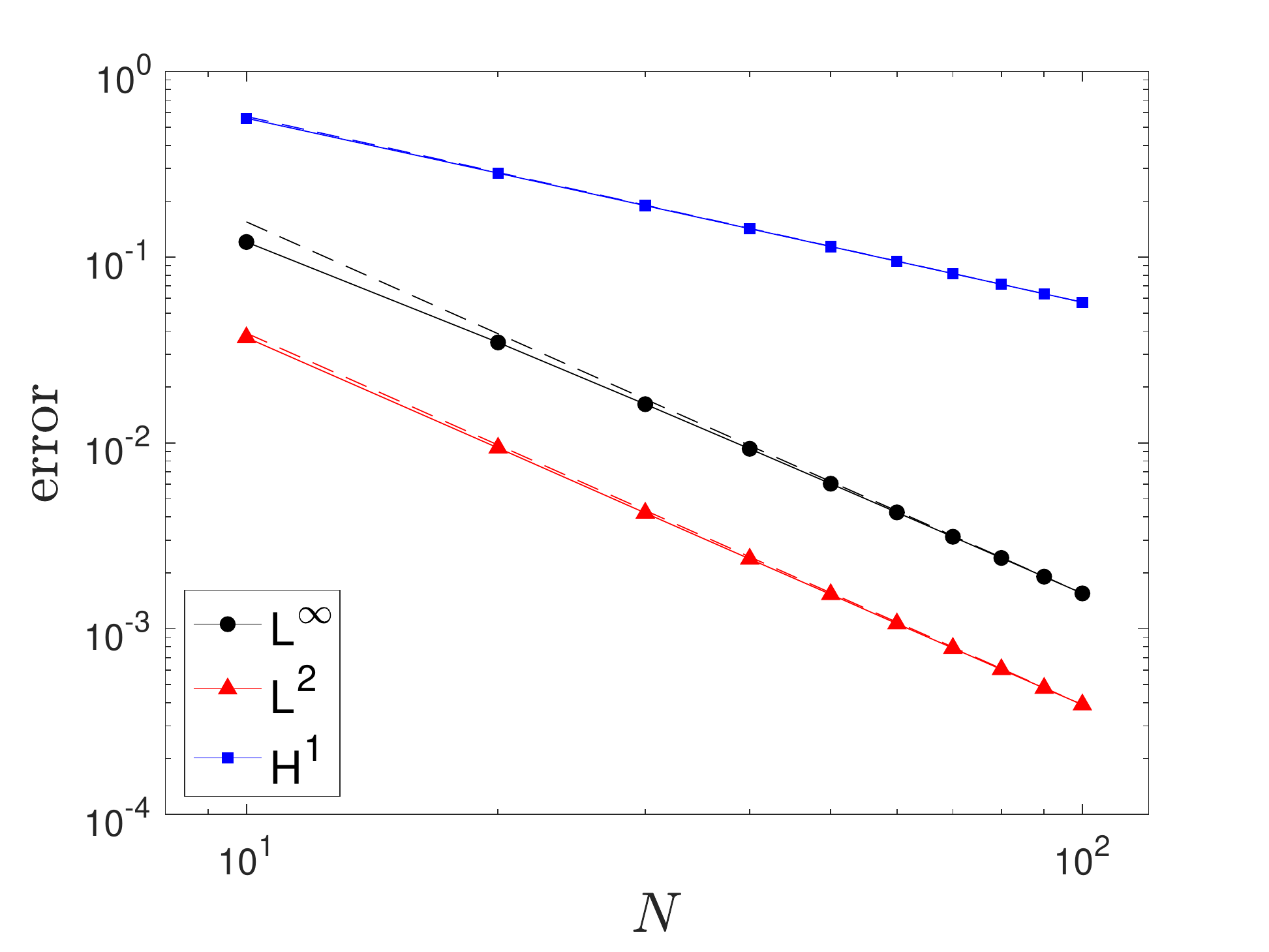}
     \label{move_cir_error} 
\end{subfigure}
~
 \begin{subfigure}{.32\textwidth}
     \includegraphics[width=2.2in]{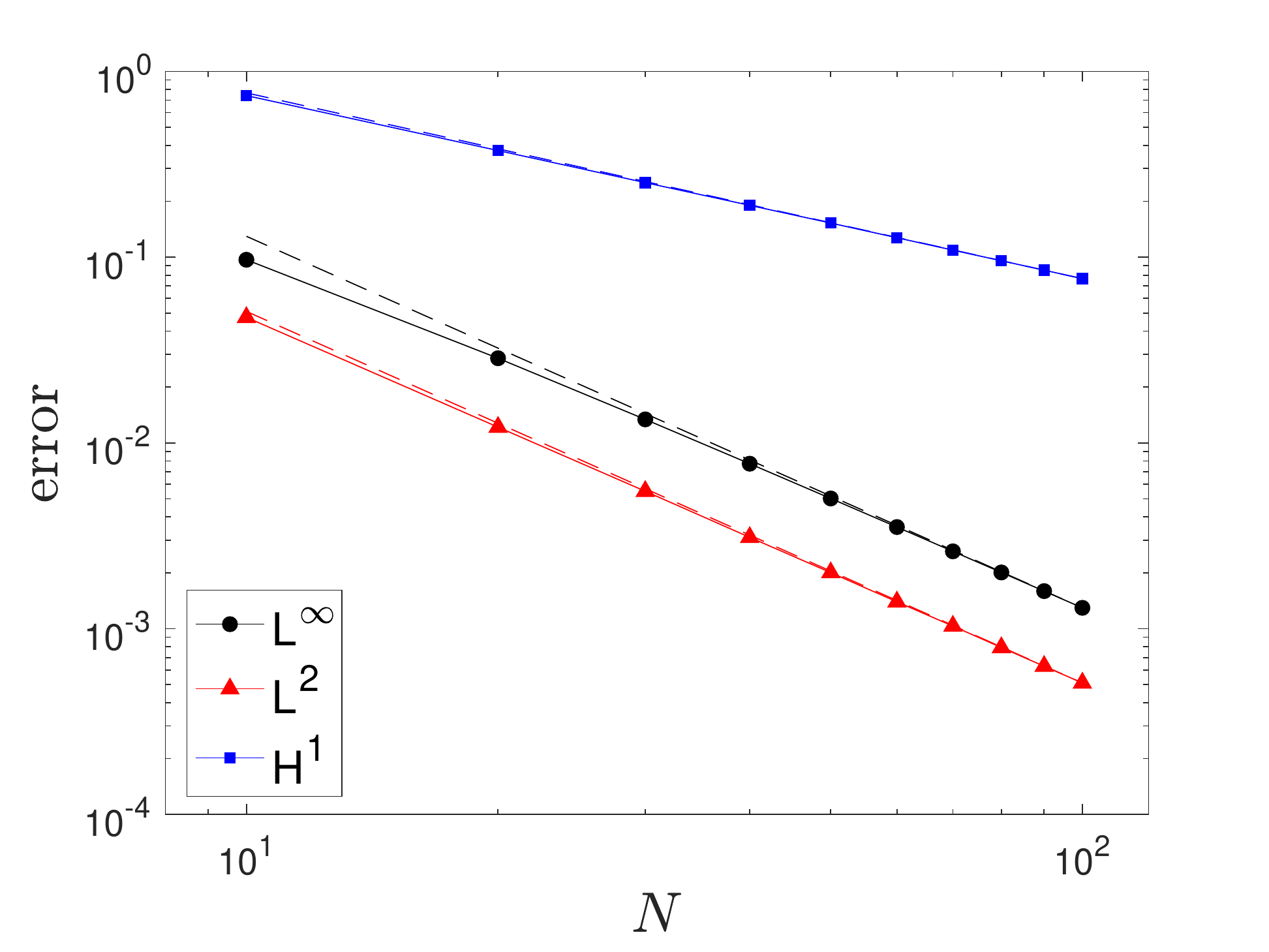}
     \label{rot_elli_error} 
\end{subfigure}
     \caption{Solution errors: a translating line (left), a moving circle (middle) and a rotating ellipse (right).}
  \label{fig:move_interf_err} 
\end{figure}

\begin{figure}[h]
\centering
\begin{subfigure}{.32\textwidth}
     \includegraphics[width=2.2in]{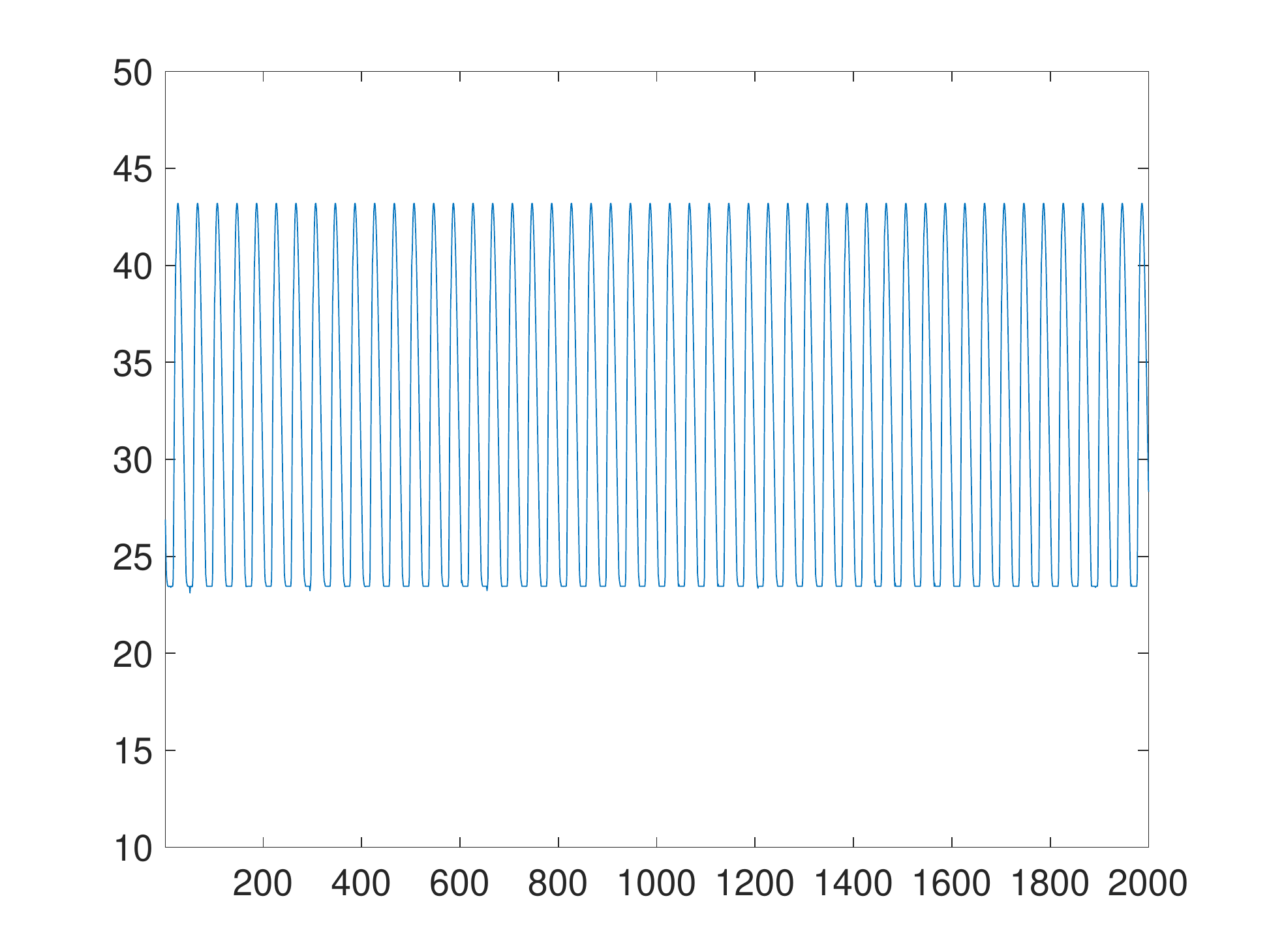}
     \label{trans_line_cond} 
\end{subfigure}
~
\begin{subfigure}{.32\textwidth}
     \includegraphics[width=2.2in]{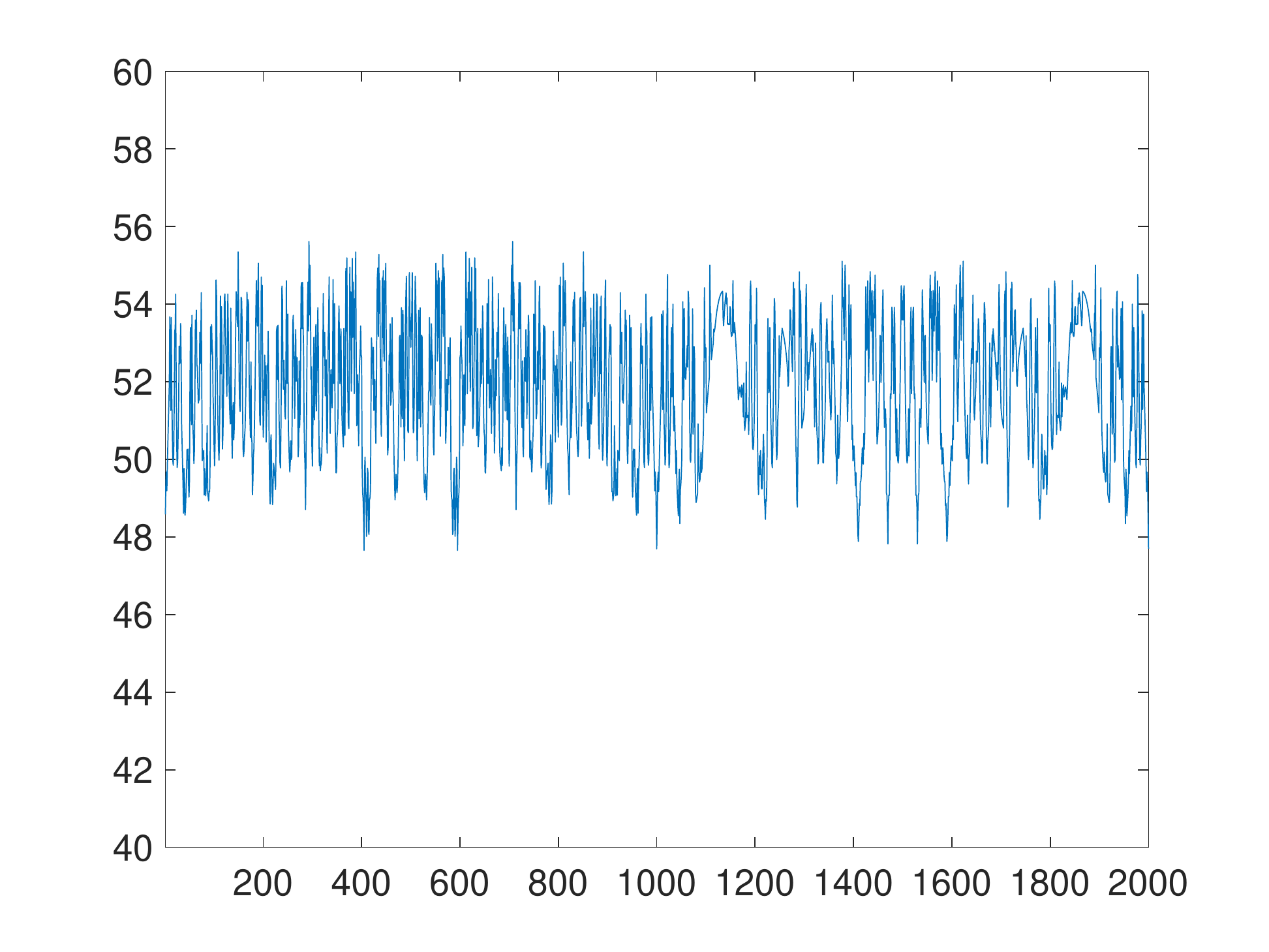}
     \label{move_cir_con} 
\end{subfigure}
~
 \begin{subfigure}{.32\textwidth}
     \includegraphics[width=2.2in]{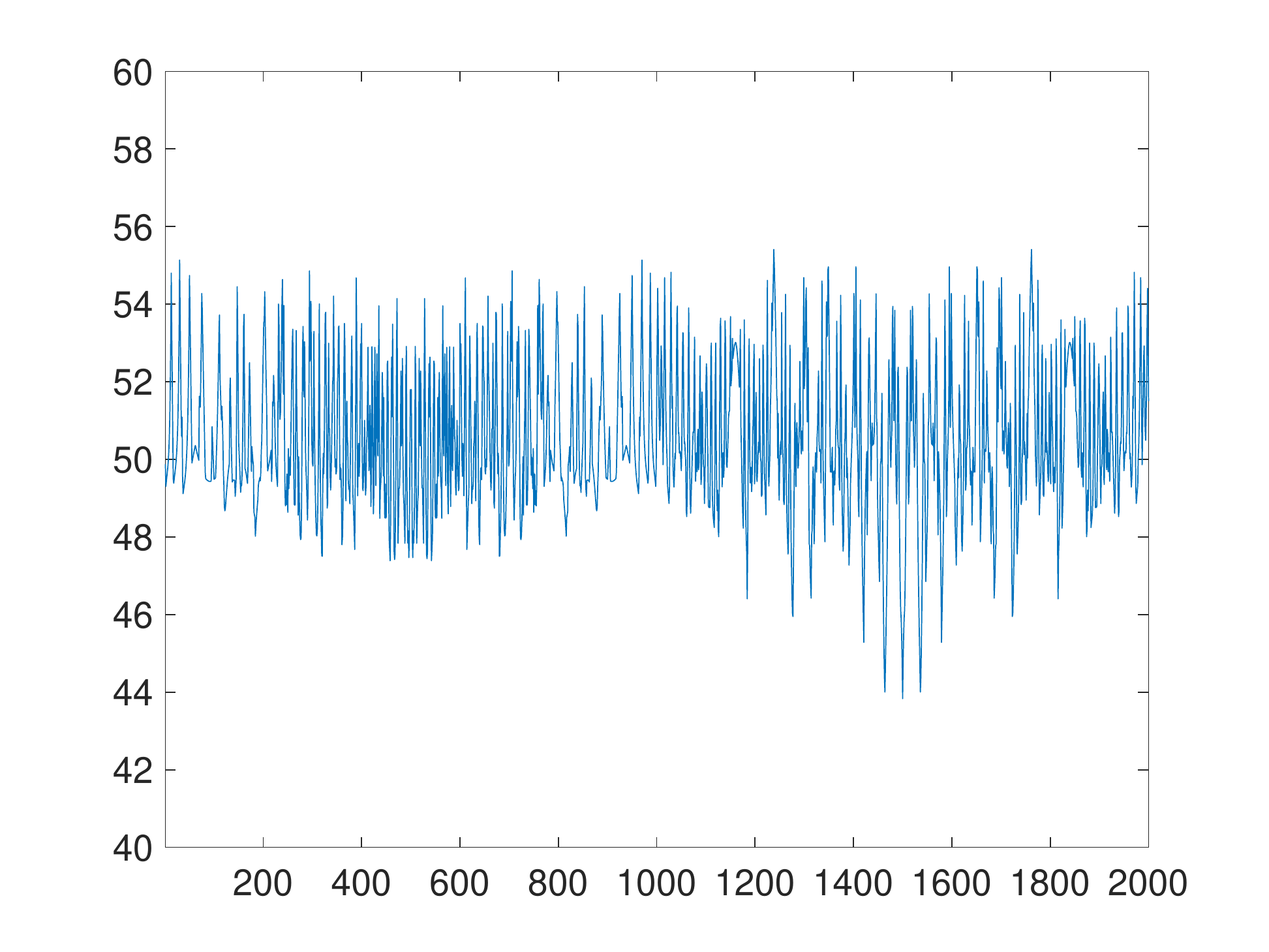}
     \label{rot_elli_cond} 
\end{subfigure}
     \caption{Condition numbers: a translating line (left), a moving circle (middle) and a rotating ellipse (right)}
  \label{fig:move_interf_cond} 
\end{figure}

We generate the mesh by partitioning $\Omega$ into $N\times N$ squares and cutting each square into two triangles. The step size is chosen as $T/N^2$ where $T$ is the total time and $N=10,20,...,100$. In Figure \ref{fig:move_interf_err}, we plot the errors of the solutions at $T=1$ gauged by the $L^{\infty}$-, $L^2$- and $H^1$-norm. For each of the error curves, we also plot a reference line which matches the ending point of the error curve and has the expected ratio, i.e., $h^{-2}$, $h^{-2}$ and $h^{-1}$ for $L^{\infty}$-, $L^2$- and $H^1$ errors, respectively. From Figure \ref{fig:move_interf_err}, we can clearly see the optimal convergence, and especially even the error in $L^{\infty}$-norm converges optimally. In particular, for the linear interface, we can see the error curves almost overlap with the reference lines. These results certainly agree with the theoretical analysis. Furthermore, for unfitted mesh methods solving moving interface problems, since the interface can be really arbitrary relative to the mesh, it is critical that condition numbers of the methods are bounded regardless of the relative location. Here to investigate this issue we also plot the condition numbers of the matrices on $N=100$ associated with the weak form $\tau a_h(u_h,v_h) + (u_h,v_h)_{L^2(\Omega)}$, $\forall u_h,v_h\in S^n_h$, $n=1,2,....$, in Figure \ref{fig:move_interf_cond}. We can clearly see all the conditions numbers are uniformly bounded during the dynamics. In particular we note that during the motion of the linear interface, at certain points the linear interface may cut all the elements with small subelements, but we can see from Figure \ref{fig:move_interf_cond} that this small-subelement issue does not cause blow-up of the condition numbers. Note that it has been theoretically addressed in \cite{2020AdjeridBabukaGuoLin} that IFE methods do not suffer from the small-subelement issue.

\bibliographystyle{plain}

\end{document}

%% file: ex_shared.tex
\usepackage{lipsum}
\usepackage{amsfonts}
\usepackage{graphicx}
\usepackage{epstopdf}
\usepackage{algorithmic}
\ifpdf
  \DeclareGraphicsExtensions{.eps,.pdf,.png,.jpg}
\else
  \DeclareGraphicsExtensions{.eps}
\fi

\newsiamremark{remark}{Remark}
\newsiamremark{hypothesis}{Hypothesis}
\crefname{hypothesis}{Hypothesis}{Hypotheses}
\newsiamthm{claim}{Claim}

\headers{An Example Article}{D. Doe, P. T. Frank, and J. E. Smith}

\title{An Example Article\thanks{Submitted to the editors DATE.
\funding{This work was funded by the Fog Research Institute under contract no.~FRI-454.}}}

\author{Dianne Doe\thanks{Imagination Corp., Chicago, IL 
  (\email{ddoe@imag.com}, \url{http://www.imag.com/\string~ddoe/}).}
\and Paul T. Frank\thanks{Department of Applied Mathematics, Fictional University, Boise, ID 
  (\email{ptfrank@fictional.edu}, \email{jesmith@fictional.edu}).}
\and Jane E. Smith\footnotemark[3]}

\usepackage{amsopn}


%% file: IFE_moving.bbl
\begin{thebibliography}{10}

\bibitem{2020AdjeridBabukaGuoLin}
Slimane Adjerid, Ivo Babu{\v s}ka, Ruchi Guo, and Tao Lin.
\newblock An enriched immersed finite element method for interface problems
  with nonhomogeneous jump conditions.
\newblock {\em Submitted,(arXiv:2004.13244)}, 2020.

\bibitem{2015AdjeridChaabaneLin}
Slimane Adjerid, Nabil Chaabane, and Tao Lin.
\newblock An immersed discontinuous finite element method for stokes interface
  problems.
\newblock {\em Comput. Methods Appl. Mech. Engrg.}, 293:170--190, 2015.
\newblock in press.

\bibitem{2018AdjeridChaabaneLinYue}
Slimane Adjerid, Nabil Chaabane, Tao Lin, and Pengtao Yue.
\newblock An immersed discontinuous finite element method for the stokes
  problem with a moving interface.
\newblock {\em Journal of Computational and Applied Mathematics}, 2018.

\bibitem{2019AdjeridMoon}
Slimane Adjerid and Kihyo Moon.
\newblock An immersed discontinuous galerkin method for acoustic wave
  propagation in inhomogeneous media.
\newblock {\em SIAM J. Sci. Comput.}, 41(1):A139--A162, 2019.

\bibitem{1993Almgren}
Robert Almgren.
\newblock Variational algorithms and pattern formation in dendritic
  solidification.
\newblock {\em Journal of Computational Physics}, 106(2):337 -- 354, 1993.

\bibitem{1982Arnold}
Douglas~N. Arnold.
\newblock An interior penalty finite element method with discontinuous
  elements.
\newblock {\em SIAM J. Numer. Anal.}, 19(4):742--760, 1982.

\bibitem{1983BabuskaOsborn}
I.~Babu{\v{s}}ka and J.~E. Osborn.
\newblock Generalized finite element methods: their performance and their
  relation to mixed methods.
\newblock {\em SIAM J. Numer. Anal.}, 20(3):510--536, 1983.

\bibitem{2000BabuskaOsborn}
Ivo Babu{\v{s}}ka and John~E. Osborn.
\newblock Can a finite element method perform arbitrarily badly?
\newblock {\em Math. Comp.}, 69(230):443--462, 2000.

\bibitem{2018BaiCaoHeLiuYang}
Jinwei Bai, Yong Cao, Xiaoming He, Hongyan Liu, and Xiaofeng Yang.
\newblock Modeling and an immersed finite element method for an interface wave
  equation.
\newblock {\em Computers \& Mathematics with Applications}, 76(7):1625--1638,
  2018.

\bibitem{2015BurmanClaus}
Erik Burman, Susanne Claus, Peter Hansbo, Mats~G. Larson, and Andr{\'e}
  Massing.
\newblock Cut{FEM}: Discretizing geometry and partial differential equations.
\newblock {\em Internat. J. Numer. Methods Engrg.}, 104(7):472--501, 2015.

\bibitem{2003CaflischLi}
Russel~E. Caflisch and Bo~Li.
\newblock Analysis of island dynamics in epitaxial growth of thin films.
\newblock {\em Multiscale Modeling \& Simulation}, 1(1):150--171, 2003.

\bibitem{1997ChenMerrimanOsherSmereka}
S.~Chen, B.~Merriman, S.~Osher, and P.~Smereka.
\newblock A simple level set method for solving {S}tefan problems.
\newblock {\em J. Comput. Phys.}, 135(1):8--29, 1997.

\bibitem{1998ChenZou}
Zhiming Chen and Jun Zou.
\newblock Finite element methods and their convergence for elliptic and
  parabolic interface problems.
\newblock {\em Numer. Math.}, 79(2):175--202, 1998.

\bibitem{2006ChrysafinosWalkington}
K.~Chrysafinos and Noel~J. Walkington.
\newblock Error estimates for the discontinuous galerkin methods for parabolic
  equations.
\newblock {\em SIAM Journal on Numerical Analysis}, 44(1):349--366, 2006.

\bibitem{2010ChuGrahamHou}
C.-C. Chu, I.~G. Graham, and T.-Y. Hou.
\newblock A new multiscale finite element method for high-contrast elliptic
  interface problems.
\newblock {\em Math. Comp.}, 79(272):1915--1955, 2010.

\bibitem{2007DiLiTangZhang}
Yana Di, Ruo Li, Tao Tang, and Pingwen Zhang.
\newblock Level set calculations for incompressible two-phase flows on a
  dynamically adaptive grid.
\newblock {\em Journal of Scientific Computing}, 31(1):75--98, 2007.

\bibitem{2001DolbowMoesBelytschko}
John Dolbow, Nicolas Mo{\"e}s, and Ted Belytschko.
\newblock An extended finite element method for modeling crack growth with
  frictional contact.
\newblock {\em Comput. Methods Appl. Mech. Engrg.}, 190(51-52):6825--6846,
  2001.

\bibitem{1991ErikssonJohnsonI}
Kenneth Eriksson and Claes Johnson.
\newblock Adaptive finite element methods for parabolic problems i: A linear
  model problem.
\newblock {\em SIAM J. Numer. Anal.}, 28(1):43--77, 1991.

\bibitem{1995ErikssonClaes}
Kenneth Eriksson and Claes Johnson.
\newblock Adaptive finite element methods for parabolic problems ii: Optimal
  error estimates in $l_{\infty}l_2$ and $l_{infty}l_{infty}$.
\newblock {\em SIAM J. Numer. Anal.}, 32(3):706--740, 1995.

\bibitem{1985ErikssonJohnsonThomee}
Kenneth Eriksson, Claes Johnson, and Vidar Thom\'ee.
\newblock Time discretization of parabolic problems by the discontinuous
  galerkin method.
\newblock {\em ESAIM Math. Model. Numer. Anal.}, 19(4):611--643, 1985.

\bibitem{2007FengKarakashian}
Xiaobing Feng and Ohannes~A. Karakashian.
\newblock Fully discrete dynamic mesh discontinuous galerkin methods for the
  cahn-hilliard equation of phase transition.
\newblock {\em Math. Comp.}, 76(259):1093--1117, 2007.

\bibitem{2017StefanThomas}
Stefan Frei and Thomas Richter.
\newblock A second order time-stepping scheme for parabolic interface problems
  with moving interfaces.
\newblock {\em ESAIM Math. Model. Numer. Anal.}, 51(4):1539--1560, 2017.

\bibitem{2016GuoLin}
Ruchi Guo and Tao Lin.
\newblock A group of immersed finite element spaces for elliptic interface
  problems.
\newblock {\em IMA J.Numer. Anal.}, 39(1):482--511, 2017.

\bibitem{2019GuoLin}
Ruchi Guo and Tao Lin.
\newblock A higher degree immersed finite element method based on a cauchy
  extension.
\newblock {\em SIAM J. Numer. Anal.}, 57(4):1545--1573, 2019.

\bibitem{2018GuoLinLinElasto}
Ruchi Guo, Tao Lin, and Yanping Lin.
\newblock Recovering elastic inclusions by shape optimization methods with
  immersed finite elements.
\newblock {\em J. Comput. Phys. (in press)}, 2019.

\bibitem{2018GuoLinZhuang}
Ruchi Guo, Tao Lin, and Qiao Zhuang.
\newblock Improved error estimation for the partially penalized immersed finite
  element methods for elliptic interface problems.
\newblock {\em Int. J. Numer. Anal. Model.}, 16(4):575--589, 2018.

\bibitem{2013HarbrechtTausch}
Helmut Harbrecht and Johannes Tausch.
\newblock On the numerical solution of a shape optimization problem for the
  heat equation.
\newblock {\em SIAM J. Sci. Comput.}, 35(1):A.104--A121, 2013.

\bibitem{2013HeLinLinZhang}
Xiaoming He, Tao Lin, Yanping Lin, and Xu~Zhang.
\newblock Immersed finite element methods for parabolic equations with moving
  interface.
\newblock {\em Numer. Methods Partial Differential Equations}, 29(2):619--646,
  2013.

\bibitem{2009HuLiTang}
Xianliang Hu, Ruo Li, and Tao Tang.
\newblock A multi-mesh adaptive finite element approximation to phase field
  models.
\newblock {\em Commun. Comput. Phys.}, 5(5):1012--1029, 2009.

\bibitem{2002HuangZou}
Jianguo Huang and Jun Zou.
\newblock Some new a priori estimates for second-order elliptic and parabolic
  interface problems.
\newblock {\em J. Differential Equations}, 184(2):570 -- 586, 2002.

\bibitem{2017HuangWuXiao}
Peiqi Huang, Haijun Wu, and Yuanming Xiao.
\newblock An unfitted interface penalty finite element method for elliptic
  interface problems.
\newblock {\em Comput. Methods Appl. Mech. Engrg.}, 323:439--460, 2017.

\bibitem{1997HuangRussel}
Weizhang Huang and Robert~D. Russell.
\newblock Analysis of moving mesh partial differential equations with spatial
  smoothing.
\newblock {\em SIAM J. Numer. Anal.}, 34(3):1106--1126, 1997.

\bibitem{1981HughesLiuZimmermann}
Thomas~J.R. Hughes, Wing~Kam Liu, and Thomas~K. Zimmermann.
\newblock Lagrangian-eulerian finite element formulation for incompressible
  viscous flows.
\newblock {\em Computer Methods in Applied Mechanics and Engineering},
  29(3):329 -- 349, 1981.

\bibitem{2020LanRamirezSun}
Rihui Lan, Michael~J. Ramirez, and Pengtao Sun.
\newblock Finite element analysis of an arbitrary {L}agrangian--{E}ulerian
  method for stokes/parabolic moving interface problem with jump coefficients.
\newblock {\em Results in Applied Mathematics}, page 100091, 2020.

\bibitem{1987Leguillon}
D.~Leguillon and E.~Sanchez-Palencia.
\newblock {\em Computation of singular solutions in elliptic problems and
  elasticity}.
\newblock Wiley, New York, 1987.

\bibitem{2013LehrenfeldReusken}
Christoph Lehrenfeld and Arnold Reusken.
\newblock Analysis of a nitsche xfem-dg discretization for a class of two-phase
  mass transport problems.
\newblock {\em SIAM Journal on Numerical Analysis}, 51(2):958--983, 2013.

\bibitem{1994LevequeLi}
Randall~J. LeVeque and Zhilin Li.
\newblock The immersed interface method for elliptic equations with
  discontinuous coefficients and singular sources.
\newblock {\em SIAM J. Numer. Anal.}, 31(4):1019--1044, 1994.

\bibitem{2010LiMelenkWohlmuthZou}
Jingzhi Li, Jens~M. Melenk, Barbara Wohlmuth, and Jun Zou.
\newblock Optimal a priori estimates for higher order finite elements for
  elliptic interface problems.
\newblock {\em Appl. Numer. Math.}, 60(1):19--37, 2010.

\bibitem{2004LiLinLinRogers}
Z.~Li, T.~Lin, Y.~Lin, and R.~C. Rogers.
\newblock An immersed finite element space and its approximation capability.
\newblock {\em Numer. Methods Partial Differential Equations}, 20(3):338--367,
  2004.

\bibitem{1997Li}
Zhilin Li.
\newblock Immersed interface methods for moving interface problems.
\newblock {\em Numer. Algorithms}, 14(4):269--293, 1997.

\bibitem{1998Li}
Zhilin Li.
\newblock The immersed interface method using a finite element formulation.
\newblock {\em Appl. Numer. Math.}, 27(3):253--267, 1998.

\bibitem{2003LiLinWu}
Zhilin Li, Tao Lin, and Xiaohui Wu.
\newblock New {C}artesian grid methods for interface problems using the finite
  element formulation.
\newblock {\em Numer. Math.}, 96(1):61--98, 2003.

\bibitem{2013LinLinZhang1}
Tao Lin, Yanping Lin, and Xu~Zhang.
\newblock A method of lines based on immersed finite elements for parabolic
  moving interface problems.
\newblock {\em Adv. Appl. Math. Mech.}, 5(4):548--568, 2013.

\bibitem{2015LinLinZhang}
Tao Lin, Yanping Lin, and Xu~Zhang.
\newblock Partially penalized immersed finite element methods for elliptic
  interface problems.
\newblock {\em SIAM J. Numer. Anal.}, 53(2):1121--1144, 2015.

\bibitem{2015LinYangZhang}
Tao Lin, Qing Yang, and Xu~Zhang.
\newblock Partially penalized immersed finite element methods for parabolic
  interface problems.
\newblock {\em Numerical Methods for Partial Differential Equations},
  31(6):1925--1947, 2015.

\bibitem{2020LinZhuang}
Tao Lin and Qiao Zhuang.
\newblock Optimal error bounds for partially penalized immersed finite element
  methods for parabolic interface problems.
\newblock {\em J. Comput. Appl. Math.}, 366:112401, 2020.

\bibitem{2001OsherFedkiw}
Stanley Osher and Ronald~P. Fedkiw.
\newblock Level set methods: An overview and some recent results.
\newblock {\em Journal of Computational Physics}, 169(2):463--502, 2001.

\bibitem{1903Reynolds}
Osborne Reynolds.
\newblock {\em Papers on mechanical and physical subjects}.
\newblock Cambridge, 1903.

\bibitem{2008Riviere}
B{\'e}atrice Rivi{\`e}re.
\newblock {\em Discontinuous {G}alerkin methods for solving elliptic and
  parabolic equations}, volume~35 of {\em Frontiers in Applied Mathematics}.
\newblock Society for Industrial and Applied Mathematics (SIAM), Philadelphia,
  PA, 2008.
\newblock Theory and implementation.

\bibitem{J.Sokolowski_J.-P.Zolesio_1992}
Jan Sokolowski and Jean-Paul Zol\'{e}sio.
\newblock {\em Introduction to shape optimization : shape sensitivity
  analysis}, volume~16 of {\em Springer Series in Computational Mathematics}.
\newblock Springer, 1992.

\bibitem{1992TezduyarBehrLiou}
T.E. Tezduyar, M.~Behr, and J.~Liou.
\newblock A new strategy for finite element computations involving moving
  boundaries and interfaces---the deforming-spatial-domain/space-time
  procedure: I. the concept and the preliminary numerical tests.
\newblock {\em Computer Methods in Applied Mechanics and Engineering},
  94(3):339 -- 351, 1992.

\bibitem{2017WangSun}
Cheng Wang and Pengtao Sun.
\newblock A fictitious domain method with distributed lagrange multiplier for
  parabolic problems with moving interfaces.
\newblock {\em Journal of Scientific Computing}, 70(2):686--716, 2017.

\bibitem{2016WangXiaoXu}
Fei Wang, Yuanming Xiao, and Jinchao Xu.
\newblock High-order extended finite element methods for solving interface
  problems.
\newblock {\em Comput. Methods Appl. Mech. Engrg.}, 364(1), 2020.

\bibitem{1967Winslow}
Alan~M. Winslow.
\newblock Numerical solution of the quasilinear poisson equation in a
  nonuniform triangle mesh.
\newblock {\em Journal of Computational Physics}, 135(2):149 -- 172, 1967.

\bibitem{2013Zunino}
Paolo Zunino.
\newblock Analysis of backward {E}uler/extended finite element discretization
  of parabolic problems with moving interfaces.
\newblock {\em Computer Methods in Applied Mechanics and Engineering}, 258:152
  -- 165, 2013.

\end{thebibliography}
